\documentclass[12pt]{article} 

\usepackage[utf8]{inputenc} 

\usepackage{geometry} 
\usepackage{float}
\usepackage{hyperref}
\usepackage{graphicx}
\usepackage{subcaption}

\usepackage{amsmath}
\usepackage{amssymb}
\usepackage{amsthm}

\usepackage{framed}

\newcommand{\jump}[1]{[\![ #1 ]\!]}

\newcommand{\llangle}{\langle\!\langle}
\newcommand{\rrangle}{\rangle\!\rangle}
\newcommand{\triple}[1]{|\!|\!| #1 |\!|\!|}

\newcommand{\bff}[1]{\mathbf{#1}}
\newcommand{\bs}[1]{\boldsymbol{#1}}

\newcommand{\Rd}{{\mathbb{R}^d}}
\newcommand{\RdG}{{\mathbb{R}^d\setminus\Gamma}}
\newcommand{\C}{\mathbb C}

\newcommand{\td}{\mathrm{TD}}

\renewcommand{\SS}{{\mathrm S}}
\newcommand{\DD}{{\mathrm D}}
\newcommand{\VV}{{\mathrm V}}
\newcommand{\WW}{{\mathrm W}}
\newcommand{\KK}{{\mathrm K}}
\newcommand{\II}{{\mathrm I}}

\newcommand{\bigO}{\mathcal{O}}

\newtheorem{proposition}{Proposition}[section]
\newtheorem{theorem}[proposition]{Theorem}

\numberwithin{equation}{section}

\iffalse 
\usepackage{tikz}
\usetikzlibrary{shapes}
\usepackage{pgfplots}
\newcommand{\includetikz}[1]{\tikzexternalenable \tikzsetnextfilename{#1}  {\include{figures/#1}} \tikzexternaldisable}
\usetikzlibrary{external}
\tikzexternaldisable

\else
\newcommand{\includetikz}[1]{\includegraphics{figures_pdf/#1}}
\fi

\title{Time-domain boundary integral equation modeling of heat transmission problems}

\author{Tianyu Qiu\footnote{Department of Mathematical Sciences, University of Delaware, Newark DE 19716, USA.  (E-mail: \href{mailto:qty@udel.edu}{\texttt{qty@udel.edu}})} \and
Alexander Rieder\footnote{Institut f\"ur Analysis und Scientific Computing, TU Wien,  A-1040 Vienna, Austria.  (E-mail: \href{mailto:alexander.rieder@tuwien.ac.at}{\texttt{alexander.rieder@tuwien.ac.at}}) Funded by the Austrian Science Fund(FWF) (grant W1245)} \and
Francisco--Javier Sayas\footnote{Department of Mathematical Sciences, University of Delaware, Newark DE 19716, USA.  (E-mail: \href{mailto:fjsayas@udel.edu}{\texttt{fjsayas@udel.edu}}). Partially funded by NSF (grant DMS 1216356)}\and
Shougui Zhang\footnote{College of Mathematics Science, Chongqing Normal University, People's Republic of China.  (E-mail: \href{shgzhang9621@sina.com}{\texttt{shgzhang9621@sina.com}}). Supported by China Scholarship Council}}
\date{\today}

\begin{document}

\maketitle

\begin{abstract}
This paper investigates the numerical modeling of a time-dependent heat transmission problem by the convolution quadrature boundary element method. It introduces the latest theoretical development into the error analysis of the numerical scheme. Semigroup theory is applied to obtain stability in spatial semidiscrete scheme. Functional calculus is employed to yield convergence in the fully discrete scheme. In comparison to the traditional Laplace domain approach, we show our approach gives better estimates. 
\end{abstract}

\section{Introduction}

Since the inception of the boundary integral equation method, the thermal engineering community has been exploiting its potential in solving transient heat conduction problems \cite{WrBr1979}. The method is also a popular choice among environmental scientists in the study of pollutant transport problem \cite{LiLi1983}. Recent applications include photothermal spectroscopy \cite{RaSa2007} and diffusion in variable media \cite{AbCaMa2013, AlRaWrSk2012}. The method enjoys sustained interest due to its remarkably simple way to handle problems on infinite domains. 

Various schemes have emerged to discretize time domain boundary integral equations associated to parabolic problems. The theoretical basis for much of this work originated with \cite{ArNo1989,Costabel1990} and focused on the development of Galerkin discretization of the basic integral equations for problems on the exterior of a bounded domain. While we will focus on a particular class of methods (Galerkin in space, Convolution Quadrature in time), let us mention that there is a large literature on other families of schemes applied to exterior problems for the heat equation \cite{GrLi2000, Tausch2007, Tausch2009, MeScTa2015, MeScTa2014}.

Our context is that of a Convolution Quadrature Boundary Element Method applied to transmission problems for the heat equation. Methods combining CQ and Galerkin BEM for heat diffusion problems appeared first in \cite{LuSc1992} (a combination of the ideas of CQ first set in \cite{Lubich1988}, with the then recent results on the single layer operator for the heat equation) and then reinterpreted as a Rothe-type method in \cite{ChKr1997}. From the point of view of the algorithm itself, we here combine a Costabel-Stephan formulation for transmission problems \cite{CoSt1985, QiSa2014}, with Galerkin semidiscretization in space and multistep or multistage CQ in time \cite{Lubich1988, LuOs1993, Banjai2010}. We set our main goals in the proof of stability and convergence of the method avoiding Laplace domain estimates (more on this later), using instead techniques of evolutionary equations associated to infinitesimal generators of strongly continuous analytic semigroups in certain Hilbert spaces. We obtain three kinds of results: (a) long term stability of the system after semidiscretization in space; (b) optimal order of convergence in time (with explicitly expressed behavior of the bounds with respect to the time variable) for BDF-based CQ schemes; (c) reduced (but higher than stage) order of convergence for RK-based CQ schemes. 

Let us next try to clarify what kind of mathematical techniques are used for our analysis and where the novelty of this work lies. To do that, we first need to discuss the mathematical background for the field. The modern analysis for time domain boundary integral equations (TDBIE) traces its roots back to the seminal paper \cite{BaHa1986}, where the bounds for boundary integral operators and layer potentials of the Helmholtz equation are derived by carrying out inversion using a Plancherel formula in anisotropic Sobolev spaces. In the regime of parabolic equations, the paper \cite{LuSc1992} converts the Laplace domain results to the time domain by bounding the inverse Laplace transform with pseudodifferential calculus, a theoretical tool not available for non-smooth domain problems. As already mentioned, \cite{ArNo1989} and \cite{Costabel1990} contain the seeds of a space-and-time coercivity analysis for the thermal boundary integral operators using anisotropic Sobolev spaces, while \cite{LuSc1992} adopts a separate strategy for the space and time variables, focusing on Sobolev regularity in space and H\"older regularity in time.
Working on a different parabolic problem (Stokes flow around a moving obstacle),  \cite{BaHaHsSa2015} gave an alternative proof of the bounds applicable to non-smooth boundaries, improving past results by revealing how all constants depend on time. Note that the usual approach was the study of mapping properties for the TDBIE and its inverse operator and for the associated potentials. Some kind of Laplace-domain coercivity was used to justify space or space-and-time Galerkin discretization. Instead, the work of \cite{LaSa2009} understood Galerkin semidiscretization in space as part of the problem set-up and its effects were analyzed as a continuous problem, instead of as a discretization of an existing continuous problem. The time-domain translation of this approach (using any variant of Laplace inversion or Payley-Wiener estimates) typically yields estimates that are suboptimal, due to the passage through the Laplace domain. (This effect can be seen for the TDBIE associated to the wave equation in \cite{Sayas2016}.) 

What we do in this paper is related to the purely time-domain analysis of TDBIE initiated in \cite{DoSa2013} for wave propagation problems. This theory has seen different extensions and refinements: for instance, \cite{QiSa2015} extends the results to the Maxwell equations and \cite{HaQiSaSa2015} is the realization that a first order in time formulation makes the analysis much simpler. The goal of exploring purely time-domain techniques is multiple. First of all, in comparison with the Laplace domain approach, and when compared on the same type of bounds (the Laplace domain also provides estimates in weighted Sobolev norms that are not available with other techniques), time-domain estimates provide: (a) lower needs for the regularity of the input data to obtain the same estimates (i.e., refined mapping properties); (b) better bounds for the estimates as time grows. The time-domain analysis also emphasizes that a dynamical process is occurring through the entire discretization problem, a process that is hidden when we focus on transfer function estimates directly attached to the boundary integral operators. 

The previous two paragraphs dealt with integral formulations, mapping properties, and semidiscretization in time. We now turn our attention to Convolution Quadrature. The multistep version of CQ, considered as a method to approximate causal convolutions and convolution equations, originated in \cite{Lubich1988}. A multistage version of the method was derived only a couple of years later in \cite{LuOs1993}. CQ techniques are now widely used in the realm of TDBIE, especially for wave propagation phenomena \cite{MoScSt2011, BaLaSa2015, LiMoWe2015}, but they are also useful in the context of TDBIE for parabolic problems \cite{LuSc1992, BaHaHsSa2015}. The Laplace domain analysis of CQ has a black-box nature that makes it very attractive: it deals with general families of operators as long as their Laplace transforms (transfer functions) satisfy certain properties. However, as already observed in the seminal work of Lubich, the CQ process applied to TDBIE can be rewritten as the application of a background ODE solver to the associated PDE in the exterior domain. In fact, rewriting the CQ process as a time-stepping procedure expressed through $\zeta$-transforms puts into evidence the fact that we are approximating a non-standard evolutionary PDE with non-homogeneous boundary conditions using an ODE solver. Along those lines, this paper offers two non-trivial contributions. First of all, we use functional calculus techniques and classical analysis of BDF methods \cite{Thomee2006} to show a direct-in-time analysis of BDF-CQ methods applied to the semidiscrete system of TDBIE. Second, we borrow heavily on difficult results by Alonso and Palencia \cite{AlPa2003} to offer an analysis of RK-CQ methods applied to the same problem. While we can prove estimates that improve the basic stage order (which is what you could obtain with direct Laplace domain analysis \cite{LuOs1993, BaLu2011, BaLuMe2011}), our numerical experiments will show that we are still slightly suboptimal and some additional work is needed.

The paper is organized as follows. Section~\ref{sect:model_problem} introduces the time domain boundary integral equation formulation (TDBIE) for the heat equation transmission problem. Section~\ref{sect:galerkin} proves the stability and convergence of the Galerkin-semidiscretization-in-space scheme. Sections~\ref{sect:multistep} and~\ref{sec:6} prove the convergence of BDFCQ and RKCQ in respectively. Finally, Section~\ref{sect:numerics} provides several numerical experiments. An appendix presents some needed background material to ease readability.

\paragraph{Notation.} For Banach spaces $X$ and $Y$, $\mathcal{B}(X,Y)$ will be used to denote the space of bounded linear operators from $X$ to $Y$. We use standard function space notations: $\mathcal C^k(I;X)$ for the space of $k$ times continuously differentiable functions of a real variable in the interval $I$ with values in the Banach space $X$, $L^2(\mathcal O)$ for the space of square integrable functions on a domain $\mathcal O$, and the Sobolev spaces
\begin{alignat*}{6}
 H^1(\mathcal O) & := \{ f\in L^2(\mathcal O) \,:\, \nabla f\in L^2(\mathcal O)^d\}, \\
 H_\Delta^1(\mathcal O)&:= \{ f\in H^1(\mathcal O)\,:\, \Delta f \in L^2(\mathcal O)\}.
\end{alignat*}
If $\Gamma$ is the boundary of a Lipschitz domain,  $H^{1/2}_\Gamma$ will be the trace space, $H^{-1/2}_\Gamma$ its dual, and $\langle \cdot, \cdot \rangle_\Gamma$ will denote the duality product of $H_\Gamma^{-1/2}\times H_\Gamma^{1/2}$. We will use the following convention
\begin{equation}\label{eq:00}
\|\,\cdot\,\|_{\mathcal O}, \qquad \|\,\cdot\,\|_{1,\mathcal O},
\qquad \| \,\cdot\,\|_{1/2,\Gamma}, \qquad \|\,\cdot\,\|_{-1/2,\Gamma},
\end{equation}
for the norms of the spaces $L^2(\mathcal O)$, $H^1(\mathcal O)$, $H^{1/2}_\Gamma$ and $H^{-1/2}_\Gamma$ respectively. We will not have a special notation for the natural norm of $H^1_\Delta(\mathcal O)$. We will use the same notation \eqref{eq:00} for the norms on Cartesian products of several copies of the same spaces.
Finally, we will denote $\mathbb R_+:=[0,\infty)$.

\section{Model problem and TDBIE Formulation}
\label{sect:model_problem}

We are concerned with a transmission problem for the heat equation in free space in presence of a single homogeneous inclusion.
 Both the inclusion and the free space medium possess homogeneous and isotropic thermal transmission properties, characterized by two positive constants: $\kappa$ as the thermal conductivity and $\rho$ as the density scaled by heat capacity. Let $\Omega_-\subset \Rd(d=2,3)$ be a bounded Lipschitz domain with boundary $\Gamma$. The normal vector field $\nu:\Gamma\to\Rd$ is defined almost everywhere on the boundary, pointing from the interior $\Omega_-$ to the exterior domain $\Omega_+:=\Rd\setminus \overline{\Omega_-}$. We can thus define two trace operators $\gamma^\pm:H^1(\RdG)\to H^{1/2}_\Gamma$, two normal derivative operators $\partial_\nu^\pm:H^1_\Delta(\RdG)\to H^{-1/2}_\Gamma$ and the jumps
\[
\jump{\gamma u}:=\gamma^-u-\gamma^+u,
	\qquad
\jump{\partial_\nu u}:=\partial_\nu^-u-\partial_\nu^+u.
\]
Given $\beta_0:[0,\infty)\to H^{1/2}_\Gamma$, and $\beta_1:[0,\infty)\to H^{-1/2}_\Gamma$, we look for $u:[0,\infty) \to H_\Delta^1(\mathbb R^d\backslash\Gamma)$ satisfying
\begin{subequations} \label{eqn:1}
\begin{alignat}{4}
\rho \dot{u} (t) &= \kappa\Delta u (t) &&\quad \text{ in } \Omega_-, &\quad \forall t\geq 0, \\
\dot{u}(t) &= \Delta u(t) &&\quad \text{ in } \Omega_+, &\quad \forall t\geq 0,\\
\gamma^- u(t) - \gamma^+ u(t) &= \beta_0(t) &&\quad \text{ on } \Gamma, &\quad \forall t \geq 0,\\
\kappa \partial_\nu^- u (t) - \partial_\nu^+ u(t) &=\beta_1(t) &&\quad \text{ on } \Gamma, &\quad \forall t \geq 0,\\
u(0) &= 0,
\end{alignat}
\end{subequations}
where upper dots denote differentiation in time.

We next give a crash course (based on \cite{Sayas2016}) on the few ingredients that are needed to have a rigorous setting for the weak definition of the heat boundary integral operators applied in the sense of distributions. Let $\mathrm F:\mathbb C_+:=\{s\in \mathbb C\,:\,\mathrm{Re}\,s>0\} \to X$ be an analytic function such that
\begin{equation}\label{eq:CCC1}		
\| \mathrm F(s)\|_{X}\leq C_{\mathrm F}(\mathrm{Re} \,s) |s|^\mu	\quad \forall s\in\mathbb{C}_+,		
\end{equation}
where $C_{\mathrm F}:(0,\infty)\to (0,\infty)$ is non-increasing and is allowed to blow-up as a rational function at the origin, i.e., there exists a constant $C>0$ and $\ell\geq 0$ such that $C_{\mathrm F}(\sigma)\leq C\sigma^{-\ell}$ when $\sigma\to 0$. It is then possible to prove \cite[Chapter 3]{Sayas2016} that there exists an $X$-valued causal tempered distribution  $f$ whose Laplace transform is $\mathrm F$,  i.e., $\mathcal L\{ f\}=\mathrm F$. The precise set of distributions whose Laplace transforms satisfy the above conditions is described in \cite[Chapter 3]{Sayas2016} and denoted $\td(X)$. These Laplace transforms (symbols, or transfer functions) include the ones that appear in parabolic problems (see \cite{BaHaHsSa2015}) where now $\mathrm F$ is well defined and analytic in $\mathbb C_\star:=\mathbb C\setminus (-\infty,0]$ and satisfies
\begin{equation}\label{eq:CCC2}
\| \mathrm F(s)\|_{X}\leq D_{\mathrm F}(\mathrm{Re} \,s^{1/2}) |s|^\mu	\quad \forall s\in\mathbb{C}_\star,
\end{equation}
where $D_{\mathrm F}:(0,\infty)\to (0,\infty)$ has the same properties as $C_{\mathrm F}$ above. Since
\[
\min\{ 1,\mathrm{Re}\, s^{1/2}\} \ge \min\{1,\mathrm{Re}\,s\} \qquad \forall s\in \mathbb C_+,
\]
a symbol satisfying \eqref{eq:CCC2} satisfies \eqref{eq:CCC1} with $C_{\mathrm F}(\sigma):=D_{\mathrm F}(\min\{1,\sigma\})$. Therefore, if $\mathrm F$ satisfies \eqref{eq:CCC2}, it is the Laplace transform of a causal distribution.

Let $s\in \mathbb C_\star$, $\phi\in H^{1/2}_\Gamma$, and $\lambda \in H^{-1/2}_\Gamma$. The transmission problem
\begin{alignat*}{6}
U\in H^1(\RdG),  \qquad & \Delta U-s U=0 \qquad &\mbox{in $\RdG$},\\
				  \jump{\gamma U}=\phi, \qquad& \jump{\partial_\nu U}=\lambda,&
\end{alignat*}
is equivalent to
\begin{subequations}\label{eq:GG1}
\begin{alignat}{6}
U\in H^1(\RdG), &\quad \jump{\gamma U} = \phi,\\
&\quad (\nabla U,\nabla V)_{\RdG} + s(U,V)_{\Rd} = \langle \lambda,\gamma V\rangle_\Gamma\quad \forall V\in H^1(\Rd),
\end{alignat}
\end{subequations}
and therefore admits a unique solution. To see that, note that the associated bilinear form is coercive as
\[
\mathrm{Re}\left(\overline s^{1/2} \left( 
	\| \nabla U\|_{\RdG}^2+ s \| U\|_{\RdG}^2\right)\right)
 =(\mathrm{Re}\, s^{1/2}) \left (\|\nabla U\|_\RdG^2+ |s| \| U\|_\RdG^2\right).
\]
This solution of \eqref{eq:GG1} can be expressed using two $s-$dependent bounded operators acting on the data:
\[
U=\SS(s)\lambda-\DD(s)\phi.
\]
This gives a simultaneous variational  definition of the single and double layer heat potentials in the Laplace domain.
By definition,
\[
\jump{\gamma}\SS(s)=0, 
	\qquad 
\jump{\partial_\nu}\SS(s) = \II,
	\qquad
\jump{\gamma}\DD(s)=-\II,
	\qquad
\jump{\partial_\nu}\DD(s)=0.
\]
We can now define the four associated boundary integral operators by taking averages of the traces and normal derivatives of the single and double layer potentials:
\begin{alignat*}{6}
& \VV(s):=\gamma^\pm \SS(s),
	&\qquad
		& \KK(s):=\tfrac12(\gamma^-+\gamma^+)\DD(s),\\
& \KK^T(s):=\tfrac12(\partial_\nu^-+\partial_\nu^+)\SS(s),
	&\qquad
		&\WW(s):=-\partial_\nu^\pm \DD(s).
\end{alignat*}
Once again by definition, the following limit relations hold:
\[
\partial_\nu^\pm \SS(s)=\mp\tfrac12\II+\KK^T(s),
	\qquad
\gamma^\pm \DD(s)=\pm\tfrac12\II+\KK(s).
\]

\begin{theorem}\label{th:bounds}
For $s\in \C_+$, denote $\sigma:=\mathrm{Re}\,s^{1/2}>0$ and $\underline\sigma:=\min\{1,\sigma\}$. There exists a constant $C$ only depending on the boundary $\Gamma$ for all $s\in \mathbb C_\star$ such that
\begin{alignat*}{6}
\| \SS(s)\|_{H^{-1/2}_\Gamma \to H^1(\Rd)} &&\leq C \frac{|s|^{1/2}}{\sigma\underline\sigma^2}, \quad &
\| \DD(s)\|_{H^{1/2}_\Gamma \to H^1(\RdG)} &&\leq C\frac{|s|^{3/4}}{\sigma\underline\sigma^{3/2}},\\
\| \VV(s)\|_{H^{-1/2}_\Gamma \to H^{1/2}_\Gamma}  && \leq C \frac{|s|^{1/2}}{\sigma\underline\sigma^2}, \quad &
\| \KK^T(s)\|_{H^{-1/2}_\Gamma \to H^{-1/2}_\Gamma}  && \leq C\frac{|s|^{3/4}}{\sigma\underline\sigma^{3/2}},\\
\| \KK(s)\|_{H^{1/2}_\Gamma \to H^{1/2}_\Gamma} &&\leq C\frac{|s|^{3/4}}{\sigma\underline\sigma^{3/2}}, \quad & \| \WW(s)\|_{H^{1/2}_\Gamma\to H^{-1/2}_\Gamma} && \leq C\frac{|s|}{\sigma\underline\sigma}.
\end{alignat*}
\end{theorem}
\begin{proof}
Since we can write the heat equation in Laplace domain as $\Delta u-(s^{1/2})^2 u=0$ for $s^{1/2}\in \mathbb{C}_+$, i.e., $s\in \mathbb C_\star$, $s$ in the estimates in \cite[Table 1]{LaSa2009} can be replaced by $s^{1/2}$.
\end{proof}

Applying the results of \cite[Chapter 3]{Sayas2016}, we can define the operator-valued distributions in the time domain through the inverse Laplace transform, using an inverse diffusivity parameter $m>0$
\begin{align*}
\mathcal S_m :=\mathcal L^{-1}\{\SS (\cdot/m) \} &\in \td(\mathcal B(H^{-1/2}_\Gamma,H^1(\Rd))),\\ 
\mathcal D_m:=\mathcal L^{-1}\{\DD (\cdot/m) \} & \in \td(\mathcal B(H^{1/2}_\Gamma,H^1(\RdG))),\\
\mathcal V_m:=\mathcal L^{-1}\{\VV(\cdot/m) \} & \in \td(\mathcal B(H^{-1/2}_\Gamma,H^{1/2}_\Gamma)), \\
\mathcal K_m:=\mathcal L^{-1}\{\KK(\cdot/m) \} & \in \td(\mathcal B(H^{1/2}_\Gamma,H^{1/2}_\Gamma)),\\
 \mathcal K^T_m:=\mathcal L^{-1}\{\KK^T(\cdot/m) \} & \in \td(\mathcal B(H^{-1/2}_\Gamma,H^{-1/2}_\Gamma)), \\
\mathcal W_m:=\mathcal L^{-1}\{\WW(\cdot/m) \} & \in \td(\mathcal B(H^{1/2}_\Gamma,H^{-1/2}_\Gamma)).
\end{align*}
When $m=1$, the subscript will be omitted. As is well known, convolutions in time correspond to multiplications in the Laplace domain.
For instance, if the Laplace transform of $\lambda \in \td (H^{-1/2}_\Gamma)$ is $\Lambda=\mathcal L\{ \lambda\}$, then $\mathcal L(\mathcal S_m*\lambda)=\SS(s/m)\Lambda (s)$ and $\mathcal S_m* \lambda \in \td(H^1(\Rd))$. The convolution operator $\lambda \mapsto \mathcal S_m * \lambda$ is the heat single layer potential. More details about the distributional convolution can be found in \cite[Section 3.2]{Sayas2016}, with a more general theory given in \cite{Treves1967}. The next theorem is the Green's representation theorem for the heat equation, which is a consequence of the analogous result in the Laplace domain.

\begin{theorem} \label{th:2}
Given $\phi\in \td(H^{1/2}_\Gamma)$ and $\lambda\in \td(H^{-1/2}_\Gamma)$, $u=\mathcal S_m*\lambda-\mathcal D_m*\phi$ is the unique solution to the problem
\begin{subequations} \label{eqn:0}
\begin{alignat}{6} 
& u\in \td(H^1_\Delta(\mathbb R^d\setminus\Gamma)) \qquad & \dot u=m\Delta u,\\
& \jump{\gamma u}=\phi,&\qquad \jump{\partial_\nu u}=\lambda.
\end{alignat}
\end{subequations}
\end{theorem}

Even though we will not need them for our numerical scheme, we now write down the explicit expression for the heat layer potentials. The time domain fundamental solution of the heat equation is
\[
\mathcal G_m(\bff{x},t):= (4\pi mt)^{-d/2} \exp \left( -{|\bff{x}|^2\over 4mt}\right) \chi_{(0,\infty)} (t).
\]
The single layer potential is then given by
\[
(\mathcal S_m* \lambda) (\bff x,t):= \int_0^t\int_\Gamma \mathcal G_m(\bff x-\bff y,t-\tau)\lambda(\bff y,\tau)\,\mathrm{d} \Gamma_{\bff y}\,\mathrm{d} \tau
\]
while the integral form of the double layer potential is
\[
(\mathcal D_m* \phi) (\bff x,t) := \int_0^t\int_\Gamma {\partial \over \partial \nu_{\bff y}} \mathcal G_m(\bff x-\bff y,t-\tau)\phi(\bff y,\tau)\,\mathrm{d} \Gamma_{\bff y}\,\mathrm{d} \tau.
\]
These integral operators are well defined for smooth enough densities $\lambda$ and $\phi$ and they coincide with the distributional defintions. Precise mapping properties in anisotropic space-time Sobolev spaces are given in the fundamental work of Martin Costabel \cite{Costabel1990}.

The transmission problem in the sense of distributions is a weak version of \eqref{eqn:1}. The data are now $\beta_0\in \td(H^{1/2}_\Gamma)$ and $\beta_1\in \td(H^{-1/2}_\Gamma)$ and we look for $u\in \td(H^1_\Delta(\mathbb R^d\setminus\Gamma)$ satisfying
\begin{subequations}\label{eqn:2}
\begin{alignat}{3}
\rho \dot{u}  & =  \kappa\Delta u, &&\quad (\text{in } L^2(\Omega_-)), \\
\dot{u} & = \Delta u, &&\quad (\text{in } L^2(\Omega_+)), \\
\gamma^- u - \gamma^+ u & =\beta_0, &&\quad (\text{in } H^{1/2}_\Gamma),\\
\kappa \partial_\nu^- u  - \partial_\nu^+ u & =\beta_1 ,&&\quad (\text{in } H^{-1/2}_\Gamma).
\end{alignat}
\end{subequations}
The upper dot is now the distributional differentiation with respect to the time variable.
The bracket in the right-hand sides of the equations in \eqref{eqn:2} clarifies where the equations hold. For instance, when we say $\rho \dot u=\kappa\Delta u$ in $L^2(\Omega_-)$, we mean that both sides of the equation are equal as $L^2(\Omega_-)$-valued distributions. A rigorous understanding of such an equation requires the elementary but careful use of steady-state operators, like distributional differentiation in the space variables or the restriction of a function to a subdomain.  

\paragraph{An integral system.} We finally derive a system of time domain boundary integral equations (TDBIE) that is equivalent to the transmission problem \eqref{eqn:2}. This follows exactly the same pattern as the work of Costabel and Stephan for steady-state (or time-harmonic) problems \cite{CoSt1985}, recently extended to transmission problems for the wave equation \cite{QiSa2014}. Since the ideas are exactly the same as in those references, we will just sketch the process. We first choose the interior trace and normal derivative of $u$ from \eqref{eqn:2} as unknowns
\begin{equation} \label{eq:2.4}
\phi := \gamma^- u,\qquad \lambda :=\partial_\nu^- u.
\end{equation}
We then define two scalar fields
\begin{subequations}\label{eq:2.7}
\begin{align}
u_-  & := \mathcal S_m * \lambda - \mathcal D_m * \phi, \quad \mbox{with $m:=\rho^{-1}\kappa$,}\\
u_+ & :=- \mathcal S * (\kappa \lambda - \beta_1)+ \mathcal D * (\phi-\beta_0),
\end{align}
\end{subequations}
each of them {\em defined on both sides of the boundary}, and related to the solution of \eqref{eqn:2} by
\[
u_- = u\chi_{\Omega_-}, \qquad u_+ = u\chi_{\Omega_+},
\]
where the symbol $\chi_{\mathcal O}$ is used to the denote the characteristic function of the set $\mathcal O$. In theory, this doubles the number of unknowns of the problem, even if we know that $u_-$ and $u_+$ vanish identically in $\Omega_+$ and $\Omega_-$ respectively.  Later on, it will  be clear that the doubling of unknowns is a natural byproduct of semidiscretization in space. The solution of \eqref{eqn:2} can be reconstructed as $u=u_-+u_+$, where $u_-,u_+\in \td(H^1_\Delta(\mathbb R^d\setminus\Gamma))$ satisfy
\begin{subequations}\label{eq:2.9}
\begin{alignat}{3}
\rho \dot{u}_-  & = \kappa\Delta u_-, &&\quad (\text{in } L^2(\RdG)), \\
\dot{u}_+ & = \Delta u_+, &&\quad (\text{in } L^2(\RdG)), \\
\gamma^+ u_- - \gamma^- u_+ & =0, &&\quad (\text{in } H^{1/2}_\Gamma), \\
\kappa \partial_\nu^+ u_-  - \partial_\nu^- u_+ & =0, &&\quad (\text{in } H^{-1/2}_\Gamma), \\
\jump{\gamma u_-} + \jump{\gamma u_+} & =\beta_0, &&\quad (\text{in } H^{1/2}_\Gamma), \\
\kappa \jump{\partial_\nu u_-}  + \jump{\partial_\nu u_+} & =\beta_1, &&\quad (\text{in } H^{-1/2}_\Gamma).
\end{alignat}
\end{subequations}
If we now substitute the representation formula \eqref{eq:2.7} in (\ref{eq:2.9}c)-(\ref{eq:2.9}d), it follows that $\lambda$ and $\phi$ satisfy
\begin{equation}\label{eq:2.8}
\left[\begin{array}{cc}
	 \mathcal V_{m}+\kappa \mathcal V & - \mathcal K_{m} - \mathcal K \\
	 \mathcal K^T_{m}+ \mathcal K^T   & \mathcal W_{m} +\frac{1}{\kappa} \mathcal W
\end{array}\right] *
\left[\begin{array}{c} \lambda \\ \phi \end{array}\right]
=
\frac12\left[\begin{array}{c} \beta_0 \\ \frac{1}{\kappa} \beta_1 \end{array}\right]
+ \left[\begin{array}{cc}
	 \mathcal V &- \mathcal K \\
	\frac{1}{\kappa} \mathcal K^T & \frac{1}{\kappa} \mathcal W
\end{array}\right] *
\left[\begin{array}{c} \beta_1 \\ \beta_0 \end{array}\right].
\end{equation}
We summarize the relations between the boundary integral equations and the partial differential equation in the following proposition. Its proof follows from elementary arguments using the jump relations of potentials and the definitions of the associate boundary integral operators.

\begin{proposition}
Assume that $(\lambda,\phi)\in \td(H^{-1/2}_\Gamma\times H^{1/2}_\Gamma)$ is a solution of \eqref{eq:2.8} and let $(u_-,u_+)\in \td(H^1_\Delta(\mathbb R^d\setminus\Gamma)^2)$ be defined by \eqref{eq:2.7}. The pair $(u_-,u_+)$ is then a solution of \eqref{eq:2.9} and $u:=u_-\chi_{\Omega_-} + u_+\chi_{\Omega_+}\in \td(H^1_\Delta(\mathbb R^d\setminus\Gamma))$ is the solution to \eqref{eqn:2}.
Reciprocally, if $u\in\td(H^1_\Delta(\mathbb R^d\setminus\Gamma))$ solves \eqref{eqn:2}, then the pair $(\lambda,\phi)$ defined by \eqref{eq:2.4} is a solution of \eqref{eq:2.8}.
\end{proposition}

\section{Galerkin semidiscretization in space}
\label{sect:galerkin}

In this section we address the semidiscretization in space of the system of TDBIE \eqref{eq:2.8}, using a completely general Galerkin scheme, and the postprocessing of the boundary unknowns using the potential expressions \eqref{eq:2.7}.

\subsection{The semidiscrete problem}

We start by choosing a pair of finite dimensional subspaces $X_h\subset H^{-1/2}_\Gamma, Y_h\subset H^{1/2}_\Gamma$. (Note that following \cite{LaSa2009} we will only need $X_h$ and $Y_h$ to be closed.) Their respective polar sets are
\begin{alignat*}{6}
X_h^\circ  & := \{ \phi\in H^{1/2}_\Gamma : \langle \mu^h, \phi\rangle_\Gamma =0\quad \forall \mu^h \in X_h \}, \\
Y_h^\circ & := \{ \lambda\in H^{-1/2}_\Gamma : \langle \lambda, \varphi^h\rangle_\Gamma =0\quad \forall \varphi^h \in Y_h \}.
\end{alignat*}
The semidiscrete method looks first for $\lambda^h \in \td(X_h), \phi^h \in \td(Y_h)$ satisfying a weakly-tested version of equations \eqref{eq:2.8}
\begin{equation} \label{eq:8}
\left[\begin{array}{cc}
	 \mathcal V_{m}+\kappa \mathcal V & - \mathcal K_{m} - \mathcal K \\
	 \mathcal K_{m}^T+ \mathcal K^T   & \mathcal W_{m} +\frac{1}{\kappa} \mathcal W
\end{array}\right] *
\left[\begin{array}{c} \lambda^h \\ \phi^h \end{array}\right]
-
\frac12\left[\begin{array}{c} \beta_0 \\ \frac{1}{\kappa} \beta_1 \end{array}\right]
- \left[\begin{array}{cc}
	 \mathcal V &- \mathcal K \\
	\frac{1}{\kappa} \mathcal K^T & \frac{1}{\kappa} \mathcal W
\end{array}\right] *
\left[\begin{array}{c} \beta_1 \\ \beta_0 \end{array}\right] \in X_h^\circ \times Y_h^\circ.
\end{equation}
The expression \eqref{eq:8} is a compacted form of the Galerkin equations: when we write that the residual of the equations is in $X_h^\circ\times Y_h^\circ$, we are equivalently requiring the residual to vanish when tested by elements of $X_h\times Y_h$. Once the boundary unknowns have been computed, the potential representation
\begin{equation}\label{eq:9}
u_-^h = \mathcal S_m * \lambda^h - \mathcal D_m * \phi^h, \qquad u_+^h =- \mathcal S * (\kappa \lambda^h - \beta_1)+ \mathcal D * (\phi^h -\beta_0),
\end{equation}
yields two fields defined on both sides of $\Gamma$ and satisfying the corresponding wave equations. 

If we subtract \eqref{eq:8} by \eqref{eq:2.8}, we obtain the system satisfied by the error of unknown densities on the boundary
\begin{equation} \label{eq:3.8}
\left[\begin{array}{cc}
	 \mathcal V_{m}+\kappa \mathcal V & - \mathcal K_{m} - \mathcal K \\
	 \mathcal K_{m}^T+ \mathcal K^T   & \mathcal W_{m} +\frac{1}{\kappa} \mathcal W
\end{array}\right] *
\left[\begin{array}{c} \lambda^h-\lambda \\ \phi^h-\phi \end{array}\right] \in X_h^\circ \times Y_h^\circ.
\end{equation}
The error corresponding to the posprocessed fields is easily derived by subtracting \eqref{eq:2.7} from \eqref{eq:9},
\begin{subequations} \label{eq:3.9}
\begin{align}
e_-^h & := u_-^h - u_- = \mathcal S_m * (\lambda^h-\lambda) - \mathcal D_m * (\phi^h-\phi), \\ 
e_+^h & := u_+^h - u_+ = - \mathcal S * \kappa ( \lambda^h - \lambda )+ \mathcal D * (\phi^h -\phi).
\end{align}
\end{subequations}

\paragraph{An exotic transmission problem.}
A transmission problem will encompass the solution of the semidiscrete system \eqref{eq:8}-\eqref{eq:9} and the associated error system \eqref{eq:3.8}-\eqref{eq:3.9}. The problem looks for $w_-,w_+\in \td(H^1_\Delta(\mathbb R^d\setminus\Gamma))$ such that 
\begin{subequations}\label{eqn:tdtransmission}
\begin{align}
\label{eqn:tdtransmissionA}
\rho \dot{w}_- & = \kappa \Delta w_-, & \dot{w}_+ & = \Delta w_+, \\
\jump{\gamma w_-} + \jump{\gamma w_+} & =\beta_0, 
&\kappa \jump{\partial_\nu w_-} +\jump{\partial_\nu w_+} & =\beta_1,\\
\gamma^+ w_- - \gamma^- w_+ & \in X_h^\circ,
&\kappa \partial_\nu^+ w_- - \partial_\nu^- w_+ & \in Y_h^\circ,\\
\jump{\gamma w_-} +\phi & \in Y_h,
&\jump{\partial_\nu w_-} +\lambda & \in X_h.
\end{align}
\end{subequations}
Equations \eqref{eqn:tdtransmissionA}  take place in $L^2(\RdG)$, while all six transmission conditions are equalities of $H_\Gamma^{\pm 1/2}$-valued distributions.
It can be shown by Theorem \ref{th:2} that \eqref{eq:8}-\eqref{eq:9} is equivalent to the above system with $\lambda=0,\phi=0$.  
On the other hand, if we set $\beta_0=0,\beta_1=0$, then the spatial semidiscrete error $(e_-^h,e_+^h)$ of \eqref{eq:3.9} is the solution to \eqref{eqn:tdtransmission}. 
 The main results for this section require some additional functional language and will be given in Section \ref{sec:3.2}, after we have embedded a stronger version of the distributional system \eqref{eqn:tdtransmission} in a framework of evolutionary problems on a Hilbert space.

\subsection{Functional framework}

The handling of the double transmission problem \eqref{eqn:tdtransmission} (with two fields defined on both sides of the interface) can be carried out with theory of differential equations associated to the infinitesimal generator of an analytic semigroup. Consider first the following spaces
\[
\bs H:=L^2(\RdG)^2,
	\qquad
\bs V:=H^1(\RdG)^2,
	\qquad
\bs D:=H^1_\Delta(\RdG)^2.
\]
To separate components of the elements of these spaces we will write $\bs w=(w_-,w_+)$. Given a constant $c\neq 0$ (we will need $c\in\{\rho,\rho^{-1},\kappa\}$) we will write $\bs c\bs w:=(c w_-,w_+)$ for the associated multiplication operator acting on the first component of the vector. We will also consider the following bilinear forms
\begin{alignat*}{6}
(\bs w,\bs w')_{\bs H}
	&:= (\bs\rho\bs w,\bs w')_\RdG 
	&&=(\rho w_-,w_-')_\RdG+(w_+,w_+')_\RdG,\\
[\bs w,\bs w']
	&:=(\bs\kappa\nabla \bs w,\nabla \bs w')_\RdG
	&&=(\kappa \nabla w_-,\nabla w_-')_\RdG+(\nabla w_+,\nabla w_+')_\RdG,
\end{alignat*}
and four copies of the boundary spaces, equipped with their product duality pairing
\[
\bs H^{\pm 1/2}_\Gamma:=(H^{\pm 1/2}_\Gamma)^4,
	\qquad
\llangle \bs n,\bs d\rrangle_\Gamma
	:=\sum_{i=1}^4 \langle n_i,d_i\rangle_\Gamma.
\]
The two-sided trace and normal derivative operators
\begin{alignat*}{6}
\bs V \ni \bs v &\quad \longmapsto \quad 
	&&\bs\gamma \bs v 
	& :=(\gamma^-v_-,\gamma^+v_-,\gamma^-v_+,\gamma^+v_+),\\
\bs D \ni \bs v &\quad \longmapsto \quad 
	&&\bs\partial_\nu \bs v 
	& :=(\partial_\nu^-v_-,\partial_\nu^+v_-,
			\partial_\nu^-v_+,\partial_\nu^+v_+),	
\end{alignat*}
are remixed to transmission operators
\[
\bs\gamma_D:=\Theta_D\bs\gamma:\bs V \to \bs H^{1/2}_\Gamma,
	\qquad
\bs\gamma_N:=\Theta_N\bs\partial_\nu:\bs D \to \bs H^{-1/2}_\Gamma,
\]
where the matrices
\[
\Theta_D:=\left[\begin{array}{cccc} 
		1 & -1 & 1 & -1 \\
		0 & 1 & -1 & 0 \\
		1 & -1 & 0 & 0 \\
		0 & 0 & 0 & 1 
		\end{array}\right],
\qquad
\Theta_N:=\left[\begin{array}{cccc} 
		1 & -1 & 1 & 0 \\
		1 & -1 & 0 & 0 \\
		0 & 1 & -1 & 0 \\
		1 & -1 & 1 & -1 
		\end{array}\right]
\]
satisfy
\[
\Theta_N^\top \Theta_D=S:=\mathrm{diag}(1,-1,1,-1).
\]
Note that the first three components of $\bs\gamma_D$ and the last three components of $\bs\gamma_N$ appear in the transmission conditions of \eqref{eqn:tdtransmission} and
\begin{alignat}{6}
\nonumber
(\nabla\bs w,\nabla\bs v)_\RdG+(\Delta\bs w,\bs v)_\RdG
& =\llangle \bs\partial_\nu \bs w,S\bs\gamma \bs v\rrangle_\Gamma 
=\llangle \Theta_N \bs\partial_\nu \bs w,\Theta_D\bs\gamma \bs v\rrangle_\Gamma\\
& =\llangle \bs\gamma_N\bs w,\bs\gamma_D \bs v\rrangle_\Gamma
\qquad \forall \bs w\in \bs D, \quad \bs v\in \bs V.
\label{eq:AA15}
\end{alignat}
This integration by parts formula can be understood as a different way of rephrasing \cite[formula (4.7)]{QiSa2014}.
We finally consider the operator
\begin{equation}
\bs D \ni \bs w \quad \longmapsto \quad 
A_\star\bs w:=\bs\rho^{-1}\bs\kappa \Delta \bs w=
(\rho^{-1} \kappa \Delta w_-,\Delta w_+) \in \bs H,
\end{equation}
and the spaces
\begin{subequations}
\begin{alignat}{6}
\bs M^{1/2} &:=\{ 0\} \times X_h^\circ \times Y_h \times H^{1/2}_\Gamma
	&& \subset \bs H^{1/2}_\Gamma,\\
\bs M^{-1/2} &:=H^{-1/2}_\Gamma \times X_h \times Y_h^\circ \times \{0\}
	&& \subset \bs H^{-1/2}_\Gamma,
\end{alignat}
\end{subequations}
which are respective polar spaces. In the coming paragraphs we will study the following problem: we are given data
\[
\bs\chi_D:[0,\infty)\to \bs H^{1/2}_\Gamma,
	\qquad
\bs\chi_N:[0,\infty)\to \bs H^{-1/2}_\Gamma,
\]
and we look for $\bs w:[0,\infty)\to \bs D$ satisfying
\begin{subequations}\label{eq:AA18}
\begin{alignat}{6}
\dot{\bs w}(t) &=A_\star \bs w(t) & \quad & \forall t\ge 0,\\
\label{eq:AA18b}
\bs\gamma_D \bs w(t)-\bs\chi_D(t) &\in \bs M^{1/2} & \quad & \forall t\ge 0,\\
\label{eq:AA18c}
\bs\gamma_N\bs\kappa \bs w(t)-\bs\chi_N(t) &\in \bs M^{-1/2} 
	& \quad & \forall t\ge 0,\\
\bs w(0)&=0.
\end{alignat}
\end{subequations}
This is a strong form (restricted to the interval $[0,\infty)$ and with strong derivatives, instead of distributional ones) of \eqref{eqn:tdtransmission} when we choose $\bs\chi_D=(\beta_0,0,-\phi,0)$ and $\bs\chi_N=(0,-\kappa\lambda,0,\beta_1)$. 
The last ingredient for our framework consists of two spaces
\begin{subequations}
\begin{alignat}{6}
\bs V_h &:=\{\bs v\in \bs V\,:\, \bs\gamma_D \bs v\in \bs M^{1/2}\},\\
\bs D_h &:=\{\bs w\in \bs D\,:\,\bs\gamma_D \bs w\in \bs M^{1/2},
\bs\gamma_N\bs\kappa\bs w\in \bs M^{-1/2}\},
\end{alignat}
\end{subequations}
and the unbounded operator $A:D(A)\to \bs H$ given by $A\bs w:=A_\star \bs w$, when $\bs w\in D(A):=\bs D_h$. In some future arguments we will find it advantageous to collect the transmission conditions \eqref{eq:AA18b}-\eqref{eq:AA18c} in a single expression, using
\[
\mathcal B\bs w:=(\bs\gamma_D\bs w,\bs\gamma_N\bs\kappa\bs w), \qquad
\bs\chi:=(\bs\chi_D,\bs\chi_N), \qquad
\bs M:=\bs M^{1/2}\times \bs M^{-1/2},
\]
so that \eqref{eq:AA18b}-\eqref{eq:AA18c} can be shortened to $\mathcal B\bs w(t)-\bs\chi(t)\in \bs M$. 

\begin{proposition}\label{prop:4}
With the above notation:
\begin{itemize}
\item[\rm (a)] For every $(\bs g,\bs\xi_D,\bs\xi_N)\in \bs H\times \bs H^{1/2}_\Gamma\times \bs H^{-1/2}_\Gamma$, the steady state problem
\begin{equation}\label{eq:BB20}
\bs w=A_\star\bs w+\bs g, 
	\qquad 
\bs\gamma_D \bs w-\bs\xi_D \in \bs M^{1/2},
	\qquad
\bs\gamma_N\bs\kappa \bs w-\bs\xi_N \in \bs M^{-1/2}
\end{equation}
admits a unique solution and
\[
\|\bs w\|_{1,\RdG}+\| \Delta \bs w\|_\RdG\le C 
	(\|\bs g\|_\RdG
		+\|\bs\xi_D\|_{1/2,\Gamma}+\|\bs\xi_N\|_{-1/2,\Gamma}).
\]
The constant $C$ depends only on $\Gamma$, $\rho$, and $\kappa$.
\item[\rm (b)] The operator $A$ is maximal dissipative and self-adjoint.
\item[\rm (c)] The operator $A$ is the generator of a contractive analytic semigroup in $\bs H$.
\end{itemize}
\end{proposition}

\begin{proof}
To prove (a), consider the coercive variational problem:
\begin{subequations}\label{eq:AA20}
\begin{alignat}{6}
& \bs w\in \bs V, \qquad \bs\gamma_D\bs w-\bs\xi_D \in \bs M^{1/2},\\
\label{eq:AA20b}
& (\bs w,\bs v)_{\bs H}+[\bs w,\bs v]
	=(\bs g,\bs w)_\RdG+\llangle \bs\xi_N, \bs\gamma_D\bs v\rrangle_\Gamma
		\quad\forall \bs v\in \bs V_h.
\end{alignat}
\end{subequations}
The coercivity constant of the bilinear form in \eqref{eq:AA20b} depends only on the constants $\kappa$ and $\rho$ if we use the standard $H^1(\RdG)^2$ norm in $\bs V_h\subset \bs V$. If we test \eqref{eq:AA20} with smooth functions that are compactly supported in $\RdG$, we can prove that $\bs\rho\bs w=\bs\kappa\Delta\bs w$. Therefore, by \eqref{eq:AA15}, it follows that
\[
\llangle \bs\gamma_N\bs\kappa \bs w-\bs\xi_N,\bs\gamma_D\bs v\rrangle_\Gamma=0 \qquad \forall \bs v\in \bs V_h,
\] 
Since $\bs\gamma_D:\bs V_h \to \bs M^{1/2}$ is surjective, this latter condition is
 equivalent to $\bs\gamma_N\bs\kappa \bs w-\bs\xi_N\in \bs M^{-1/2}=(\bs M^{1/2})^\circ$.

Note next that
\[
\llangle \bs\gamma_N\bs\kappa\bs w,\bs\gamma_D \bs v\rrangle_\Gamma=0
	\qquad\forall \bs w \in \bs D_h, \quad\bs v\in \bs V_h,
\]
and therefore, by \eqref{eq:AA15},
\begin{equation}\label{eq:BB22}
(A\bs w,\bs v)_{\bs H}=-[\bs w,\bs v] \quad\forall \bs w \in \bs D_h, \quad\bs v\in \bs V_h.
\end{equation}
This proves symmetry and dissipativity of $A$. Taking $\bs\xi_D=0$ and $\bs\xi_N=0$ in (a), we easily show that $I-A:D(A)\to \bs H$ is surjective and, therefore, $A$ is maximal dissipative and self-adjoint (see \cite[Proposition 3.11]{Schmudgen2012}). 

Finally $A$ is the infinitesimal generator of a contractive semigroup if and only if it is maximal dissipative (see \cite[Chapter 1, Theorem 4.3]{Pazy1983} or \cite[Theorem 4.4.3, Theorem 4.5.1]{Kesavan1989}) and the dissipativity and self-adjointness of $A$ show that the semigroup is analytic (see \cite[Corollary 4.8]{EnNa2006}).
\end{proof}

If we define
\[
| \bs w|_V:=[\bs w,\bs w]^{1/2} \approx \|\nabla \bs w\|_\RdG,
\]
the identity \eqref{eq:BB22} and a simple computation show that
\begin{equation}\label{eq:NEW24}
|\bs w|_V \le \| A\bs w\|_H + \|\bs w\|_H \approx
\|\Delta \bs w\|_\RdG+\|\bs w\|_\RdG \qquad \forall \bs w\in D(A).
\end{equation}
In the sequel we will also use H\"{o}lder spaces 
$\mathcal C^\theta(\mathbb{R}_+;X)$ for $\theta \in (0,1)$, where $X$ is a Hilbert space, and the seminorms
\[
| f|_{t,\theta,X} := \sup_{0\leq \tau_1 <  \tau_2\leq t} {\| f (\tau_1)- f(\tau_2)\|_X \over |\tau_1-\tau_2|^\theta},\qquad t>0.
\]

\begin{proposition}\label{prop:333}
Let $\bs\chi:=(\bs\chi_D,\bs\chi_N):\mathbb R_+\to \bs H_\Gamma:=\bs H^{1/2}_\Gamma\times \bs H^{-1/2}_\Gamma$ and assume that
\[
\dot{\bs\chi}\in \mathcal C^\theta(\mathbb R_+;\bs H_\Gamma),
	\qquad \bs\chi(0)=\dot{\bs\chi}(0)=0.
\]
The problem \eqref{eq:AA18} admits a unique solution satisfying
\[
\dot{\bs w}= A_\star\bs w\in \mathcal C^\theta(\mathbb R_+;\bs H),
	\qquad
\bs w\in \mathcal C^\theta(\mathbb R_+;\bs V).
\]
Moreover, there exist constants independent of $h$ such that for all $t$
\begin{subequations}
\begin{alignat}{6}
\label{eq:26a}
\| \bs w(t)\|_\RdG 
	& \le C_1 t (\|\bs\chi\|_{L^\infty(0,t;\bs H_\Gamma)}
			+\|\dot{\bs\chi}\|_{L^\infty(0,t;\bs H_\Gamma)}),\\
\label{eq:26b}
\| \Delta\bs w(t)\|_\RdG
	&\le C_2 \left(\max\{1,t\} \|\dot{\bs\chi}\|_{L^\infty(0,t;\bs H_\Gamma)}
		+\theta^{-1} t^\theta |\dot{\bs\chi}|_{t,\theta,\bs H_\Gamma}\right),\\
\label{eq:26c}
\| \nabla \bs w(t)\|_\RdG
	& \le C_3 \left( t\|\bs\chi\|_{L^\infty(0,t;\bs H_\Gamma)}
		+\max\{1,t\} \|\dot{\bs\chi}\|_{L^\infty(0,t;\bs H_\Gamma)}
		+\theta^{-1} t^\theta |\dot{\bs\chi}|_{t,\theta,\bs H_\Gamma}\right).
\end{alignat}
\end{subequations}
\end{proposition}

\begin{proof}
The proof is based on the decomposition of the solution of \eqref{eq:AA18} into a sum $\bs w=\bs w_\chi+\bs w_0$, where $\bs w_\chi$ takes care of the data (using Proposition \ref{prop:4}), while $\bs w_0$ will be handled using a Cauchy problem (Theorem \ref{th:semigroup}). 

Let $\bs w_\chi(t)$ be the solution of \eqref{eq:BB20} with $\bs g=0$, $\bs\xi_D=\bs\chi_D(t)$ and $\bs\xi_N=\bs\chi_N(t)$, and note that $\dot{\bs w}_\chi=\bs w_{\dot\chi}\in \mathcal C^\theta(\mathbb R_+;\bs D)$, since we have applied a time-independent operator to the transmission data. Let now  
$\bs f:=\bs w_\chi-\dot{\bs w}_\chi\in \mathcal C^\theta(\mathbb R_+;\bs H),$ which satisfies $\bs f(0)=0.$
Note that for all $t\ge 0$
\begin{subequations}
\begin{alignat}{6}
\label{eq:27a}
\|\bs w_\chi(t)\|_{1,\RdG}+\| \Delta\bs w_\chi(t)\|_\RdG
	&\le C \|\bs \chi(t)\|_{\bs H_\Gamma}, \\
\label{eq:27b}
\| \bs f(t)\|_\RdG
	& \le C (\|\bs \chi(t)\|_{\bs H_\Gamma}
		+\| \dot{\bs\chi}(t)\|_{\bs H_\Gamma}),\\
\label{eq:27c}
| \bs f|_{\theta,t,\bs H}
 	&\le C( |\bs\chi|_{\theta,t,\bs H_\Gamma}+ 
 	|\dot{\bs\chi}|_{\theta,t,\bs H_\Gamma}),
\end{alignat}
\end{subequations}
with a constant $C$ depending exclusively on the parameters and geometry. Let finally $\bs w_0:\mathbb R_+\to D(A)$ be the solution of
\[
\dot{\bs w}_0(t)=A\bs w(t)+\bs f(t) \qquad t\ge 0, \qquad \bs w(0)=0.
\]
By Theorem \ref{th:semigroup}, $\bs w_0$ and therefore $\bs w$ have the required regularity. Using the bound for $\bs w_0$ in Theorem \ref{th:semigroup} and \eqref{eq:27a}-\eqref{eq:27b}, we can easily prove \eqref{eq:26a}. Using the bound for $A\bs w_0$ in Theorem \ref{th:semigroup} and \eqref{eq:27a}-\eqref{eq:27c}, we can prove that
\[
\|\Delta \bs w(t)\|_\RdG \le 
C \left(\|\bs\chi(t)\|_{\bs H_\Gamma}
+ \|\dot{\bs\chi}(t)\|_{\bs H_\Gamma}
+\theta^{-1} t^\theta |\bs\chi|_{t,\theta,\bs H_\Gamma}
+\theta^{-1} t^\theta |\dot{\bs\chi}|_{t,\theta,\bs H_\Gamma}\right).
\]
Proving \eqref{eq:26b} from the above estimate is the result of a simple computation.  Finally, in view of \eqref{eq:NEW24}, the estimate
\begin{alignat*}{6}
\|\nabla\bs w(t)\|_\RdG & \le 
C(\|\nabla\bs w_\chi(t)\|_\RdG + \| \bs w_0(t)\|_\RdG+\| A\bs w_0(t)\|_\RdG) \\
& \le C' (\|\bs w_\chi(t)\|_{1,\RdG}
	+\| \bs w(t)\|_\RdG+\|\Delta\bs w(t)\|_\RdG)
\end{alignat*}
follows and therefore \eqref{eq:26c} is a simple consequence of \eqref{eq:27a}, \eqref{eq:26a}, and \eqref{eq:26b}.
\end{proof}

\subsection{Main results on the semidiscrete problem}\label{sec:3.2}

The first step towards the analysis of the two problems that are hidden in \eqref{eqn:tdtransmission} is the reconciliation of the solution of the classical differential equation \eqref{eq:AA18} with a distributional form, where we look for $\bs w\in \td(\bs D)$ such that
\begin{subequations}\label{eq:BBB100}
\begin{alignat}{6}
\dot{\bs w} &=A_\star \bs w & \qquad & (\mbox{in $\bs H$}),\\
\mathcal B\bs w-\bs\eta
&\in \bs M & & (\mbox{in $\bs H_\Gamma$}),
\end{alignat}
\end{subequations}
for given data $\bs\eta\in \td(\bs H_\Gamma)$. 

\begin{proposition}
Problem \eqref{eq:BBB100} has a unique solution.
\end{proposition}

\begin{proof}
Let $\mathrm H=(\mathrm H_D,\mathrm H_N):=\mathcal L\{\bs\eta\}$. The $s$-dependent transmission problem 
\begin{subequations}\label{eq:3131}
\begin{alignat}{6}
s\bs W(s) & =A_\star \bs W(s),\\
\mathcal B \bs W(s)-\mathrm H(s)
& \in \bs M,
\end{alignat}
\end{subequations}
is equivalent to the variational problem
\begin{subequations}\label{eq:3232}
\begin{alignat}{6}
& \bs W(s)\in \bs V,\\
& \bs \gamma_D \bs W(s)-\mathrm H_D(s)\in \bs M^{1/2},\\
& [\bs W(s),\bs w]+s(\bs V(s),\bs w)_{\bs H}=
	\llangle \mathrm H_N(s),\bs\gamma_D\bs w\rrangle_\Gamma
	\qquad\forall \bs w\in \bs V_h,
\end{alignat}
\end{subequations}
for all $s\in \mathbb C_+$,
which can be easily proved using the techniques of the proof of Proposition \ref{prop:4}. We will prove that \eqref{eq:3232} is uniquely solvable and that its solution can be bounded as
\begin{equation}\label{eq:3333}
\| \bs W(s)\|_\RdG+\| \nabla \bs W(s)\|_\RdG \le C(\mathrm{Re}\,s) |s|^\nu \, \| \mathrm H(s)\|_{\bs H_\Gamma} \qquad \forall s\in \mathbb C_+,
\end{equation}
for some $\nu \ge 0$ and non-increasing $C:(0,\infty)\to (0,\infty)$ that is allowed to grow rationally at the origin. These statements imply that $\bs W=\mathcal L\{\bs w\}$ where $\bs w\in \td(\bs D)$ (note that the needed bounds for the Laplacian of $\bs W(s)$ follow from equation \eqref{eq:3131}) and $\bs w$ satisfies \eqref{eq:BBB100}, which is the inverse Laplace transform of \eqref{eq:3131}.

In order to deal with \eqref{eq:3232} and \eqref{eq:3333}, we proceed as follows. For fixed $s\in \mathbb C_+$, we consider the coercive transmission problem
\begin{subequations}\label{eq:3434}
\begin{alignat}{6}
 -|s|\bs W_D(s)+A_\star\bs W_D(s) & =0, \\
 \label{eq:3434b}
 \bs \gamma \bs W_D(s) &=\Theta_D^{-1} \mathrm H_D(s),
\end{alignat}
\end{subequations}
and note that the four separate boundary conditions in \eqref{eq:3434b} are equivalent to the transmission conditions $\bs\gamma_D \bs W_D(s)=\mathrm H_D(s)$. Using the Bamberger-HaDuong lifting lemma (the original appears in \cite{BaHa1986} and an `extension' to non-smooth boundaries can be found as Lemma 2.7.1 in \cite{Sayas2016}), it follows that
\begin{alignat}{6}
\nonumber
\triple{\boldsymbol W_D(s)}_{|s|}^2
& := \|\boldsymbol\kappa^{1/2}\nabla \boldsymbol W_D(s)\|_\RdG^2+ |s|\, \|\boldsymbol\rho^{1/2} \boldsymbol W_D(s)\|_\RdG^2 \\
\label{eq:3434B}
& \le C \,\left( \frac{|s|}{\min\{1,\mathrm{Re}\,s\}}\right)^{1/2} \| H_D(s)\|^2_{1/2,\Gamma}.
\end{alignat}
We then consider the coercive variational problem
\begin{subequations}
\begin{alignat}{6}
& \boldsymbol W_0(s)\in \boldsymbol V_h,\\
\label{eq:3535b}
& [\boldsymbol W_0(s),\boldsymbol w]+s(\boldsymbol W_0(s),\boldsymbol w)_{\boldsymbol H}
	= && \llangle \mathrm H_N(s),\boldsymbol\gamma_D\boldsymbol w\rrangle_\Gamma \\
	&&& -[\boldsymbol W_D(s),\boldsymbol w]
	-s(\boldsymbol W_D(s),\boldsymbol w)_{\boldsymbol H} \quad \forall \boldsymbol w\in \boldsymbol V_h.
\end{alignat}
\end{subequations}
Testing \eqref{eq:3535b} with $\boldsymbol w=\overline{s^{1/2}\boldsymbol W_0(s)}$, we have the inequalities
\begin{alignat*}{6}
(\mathrm{Re}\,\overline s^{1/2}) \triple{\boldsymbol W_0(s)}_{|s|}^2
	&=\mathrm{Re}\,
		\left(\overline s^{1/2}\left( [\boldsymbol W_0(s),\overline{\boldsymbol W_0(s)}]
			+s (\boldsymbol W_0(s),\overline{\boldsymbol W_0(s)})_{\boldsymbol H}\right)\right) \\
	& \le (\mathrm{Re}\, s^{1/2}) 
		\left( \|\mathrm H_N(s)\|_{-1/2,\Gamma}\| \boldsymbol\gamma_D\boldsymbol W_0(s)\|_{1/2,\Gamma}
			+ \triple{\boldsymbol W_D(s)}_{|s|}\triple{\boldsymbol W_0(s)}_{|s|}\right).
\end{alignat*}
What is left for the proof is very simple indeed. First of all, it is clear that $\bs W(s):=\bs W_D(s)+\bs W_0(s)$ is the solution to \eqref{eq:3232}. Second, it is simple to see that
\[
\|\bs\gamma_D\bs W_D(s)\|_{1/2,\Gamma}
\le \frac{C}{\min\{1,\mathrm{Re}\,s^{1/2}\}}\triple{\bs W_D(s)}_{|s|}
\le \frac{C}{\min\{1,\mathrm{Re}\,s\}}\triple{\bs W_D(s)}_{|s|},
\]
which yields a bound for $\triple{\bs W_0(s)}_{|s|}$ in terms of $\|\mathrm H_N(s)\|_{-1/2,\Gamma}$ and $\triple{\bs W_D(s)}_{|s|}$. Finally \eqref{eq:3434B} can be used to prove \eqref{eq:3333}.
\end{proof}

The following process mimics the one in \cite{HaQiSaSa2015} and in \cite[Chapter 7]{Sayas2016}. It involves two aspects: (a) an extension of zero of the data to negative values of the time variable; (b) a hypothesis on polynomial growth of the data. The reason to deal with (a) lies in the fact that the distributional equations \eqref{eq:BBB100} are for causal distributions of the real variable, not for distributions defined in the positive real axis. This is due to the fact that the heat potentials and operators have memory terms that involve the entire history of the process including values at time $t=0$. The extension by zero to negative time will be done through the operator
\[
E f(t)=\begin{cases} f(t), & t\ge 0, \\ 0, & t< 0.\end{cases}
\]
The reason why (b) is important is the fact that the equation \eqref{eq:BBB100} can be shown to have a unique solution in the space $\td(\bs D)$, which imposes some restrictions on the growth of the solution (and hence the data) at infinity. 

\begin{proposition}\label{prop:55}
Let $\bs\chi:\mathbb R_+\to \bs H_\Gamma$ be continuous and $\dot{\bs\chi}\in \mathcal C^\theta(\mathbb R_+;\bs H_\Gamma)$ be polynomially bounded in the following sense: there exist $C>0$ and $m\ge 0$ such that
\[
\|\dot{\bs\chi}(t)\|_{\bs H_\Gamma}+|\dot{\bs\chi}|_{t,\theta,\bs H_\Gamma}
\le C t^m \qquad \forall t\ge 1.
\]
Assume also that $\bs\chi(0)=\dot{\bs\chi}(0)=0$. 
If $\bs w$ is the solution to \eqref{eq:AA18}, then $\bs v=E\bs w$ is the solution to \eqref{eq:BBB100} with $\bs\eta=E\bs\chi$.
\end{proposition}

\begin{proof}
The hypotheses imply that $E\bs\chi\in \td(\bs H_\Gamma)$. The bounds of Proposition \ref{prop:333} imply that $\bs w$ is polynomially bounded as a $\bs D$-valued function and therefore $\bs v:=E\bs w\in \td(\bs D)$. Finally, since $\bs w(0)=0$, it follows that $E\dot{\bs w}=\dot{\bs v}$, which finishes the proof, since $E$ commutes with any operator that does not affect the time variable.
\end{proof}

We are almost ready to state and prove the two main results concerning the semidiscrete system: semidiscrete stability and an error estimate. To shorten up some of the expressions, to come, we introduce the bounded jump operator
\[
\bs D \ni \bs u=(u_-,u_+)
\longmapsto \bs J\bs u:=(\jump{\gamma u_-}, \jump{\partial_\nu u_-})\in H^{1/2}_\Gamma\times H^{-1/2}_\Gamma=:H_\Gamma.
\]
We also consider the function spaces tagged in the parameter $\theta\in (0,1)$
\begin{alignat*}{6}
\mathcal B^\theta :=\{ \bs\eta\in \mathcal C(\mathbb R;H_\Gamma)\,:\,
	& \bs\eta\equiv 0 \mbox{ in $(-\infty,0)$}, \quad
	 \bs\eta|_{\mathbb R_+}\in \mathcal C^\theta(\mathbb R_+;H_\Gamma),\\
	& \exists C,m\ge 0 \mbox{ s.t. }
	\| \bs\eta(t)\|_{H_\Gamma}+|\bs\eta|_{t,\theta,H_\Gamma}\le C t^m \quad \forall t\ge 1\},\\
\mathcal B^{1+\theta} :=\{ \bs\eta:\mathbb R \to H_\Gamma\,:\,
	& \bs\eta\equiv 0 \mbox{ in $(-\infty,0)$}, \quad \dot{\bs\eta}\in \mathcal B^\theta\},\\
\mathcal U^\theta :=\{\bs u\in \mathcal C^1(\mathbb R;\bs H)\,:\,
	& \bs u\equiv 0 \mbox{ in $(-\infty,0)$}, \quad 
	\bs u|_{\mathbb R_+}\in \mathcal C^\infty(\mathbb R_+;\bs D)\}.
\end{alignat*}

\begin{theorem}\label{the:3.6}
If $\bs\beta\in \mathcal B^{1+\theta}$, $\bs\psi^h=(\phi^h,\lambda^h)$ is the solution to the semidiscrete system of TDBIE \eqref{eq:8} and $\bs u^h=(u^h_-,u^h_+)$ is given by the potential representation \eqref{eq:9}, then $\bs u^h\in \mathcal U^\theta$, $\bs\psi^h\in \mathcal B^\theta$, and we can bound
\[
\|\bs u^h(t)\|_{1,\RdG}+\|\bs \psi^h(t)\|_{H_\Gamma}
\le C \max\{1,t\} \mathrm{Acc}(\bs\beta,t,\theta),
\]
where
\begin{equation}
\mathrm{Acc}(\bs\beta,t,\theta)
:=\max_{0\le \tau\le t}\|\bs\beta(\tau)\|_{H_\Gamma}
+\max_{0\le \tau\le t}\|\dot{\bs\beta}(\tau)\|_{H_\Gamma}
+ \frac{t^\theta}\theta\sup_{0\le \tau_1<\tau_2\le t}
\frac{\| \bs\beta(\tau_1)-\bs\beta(\tau_2)\|_{H_\Gamma}}{|\tau_1-\tau_2|^\theta},
\end{equation}
is a collection of cummulative seminorms in $\mathcal B^{1+\theta}$.
\end{theorem}

\begin{proof}
If $\bs\beta=(\beta_0,\beta_1)$, we define $\bs\chi_D:=(\beta_0,0,0,0)$ and $\bs\chi_N=(0,0,0,\beta_1)$. Let $\bs u^h$ be the solution to \eqref{eq:BBB100} with data $\bs\eta=(\bs\chi_D,\bs\chi_N)$ and $\bs\psi^h:=\bs J\bs u^h$. We can identify $\bs u^h$ and $\bs  \psi^h$  with the solution of \eqref{eq:8} and \eqref{eq:9}. We also note that
\[
\|\bs u^h(t)\|_{1,\RdG}+\|\bs \psi^h(t)\|_{H_\Gamma}
\le C( \|\bs u^h(t)\|_\RdG+\|\nabla\bs u^h(t)\|_\RdG+\|\Delta\bs u^h(t)\|_\RdG).
\]
The bounds in the statement of the theorem follow from the fact that Proposition \ref{prop:55} identifies $\bs u^h|_{\mathbb R_+}$ with the solution of \eqref{eq:AA18}, with $(\bs\chi_D,\bs\chi_N)|_{\mathbb R_+}$ as data, and we can thus use the estimates of Proposition \ref{prop:333}.
\end{proof}

\begin{theorem}
For data $(\beta_0,\beta_1)\in \td(H_\Gamma)$, we let
\begin{itemize}
\item[\rm (a)] $\bs\psi=(\phi,\lambda)$ be the solution of the TDBIE \eqref{eq:2.8},
\item[\rm (b)] $\bs u=(u_-,u_+)$ be given by the potential representation \eqref{eq:2.7},
\item[\rm (c)] $\bs\psi^h=(\phi^h,\lambda^h)$ be the solution of the semidiscrete TDBIE \eqref{eq:8},
\item[\rm (d)] $\bs u^h=(u^h_-,u^h_+)$ be given by the potential representation \eqref{eq:9}.
\end{itemize}
If $\bs\psi\in \mathcal B^{1+\theta}$, then
\begin{equation}\label{eq:3.AA}
\| \bs u^h(t)-\bs u(t)\|_{1,\RdG}+\|\bs\psi^h(t)-\bs\psi(t)\|_{H_\Gamma}
\le C \max\{1,t\} \mathrm{Acc}(\bs\psi-\bs\Pi_h\bs\psi,t,\theta),
\end{equation}
where $\bs\Pi_h:H_\Gamma\to Y_h\times X_h$ is the best approximation operator onto $Y_h\times X_h$.
\end{theorem}

\begin{proof}
First of all, note that the errors $\bs\psi^h-\bs\psi$ can be defined as the solution of the semidiscrete TDBIE \eqref{eq:3.8} and the associated potential errors $\bs e^h:=\bs u^h-\bs u$ are given by a potential representation \eqref{eq:3.9} using $\bs\psi^h-\bs\psi$ as input densities. 
If we define $\boldsymbol\chi_D=(0,0,-\phi,0)$ and $\bs\chi_N=(0,-\kappa\lambda,0,0)$, we can see that $\bs e^h$ is the solution to \eqref{eq:BBB100} with data $\bs\eta=(\bs\chi_D,\bs\chi_N)$ and that $\bs\psi^h-\bs\psi=\bs J\bs e^h$. Using the same arguments as in the proof of Theorem \ref{the:3.6}, we can prove that
\begin{equation}\label{eq:3.AB}
\|\bs e^h(t)\|_{1,\RdG}+\|\bs J\bs e^h(t)\|_{H_\Gamma}\le C \max\{1,t\} \mathrm{Acc}(\bs\psi,t,\theta).
\end{equation}
Note now that we can decompose $\bs\psi^h-\bs\psi=(\bs\psi^h-\bs\Pi_h\bs\psi)-(\bs\psi-\bs\Pi_h\bs \psi)$ and consider $\bs\psi-\bs\Pi_h\bs \psi$ as the exact solution in the argument and $\bs\psi^h-\bs\Pi_h\bs\psi$ as its Galerkin approximation. In other words, if we input $\bs\Pi_h\bs\psi$ as exact solution of the TDBIE, then $\bs\Pi_h\bs\psi$  is also the solution of the semidiscrete TDBIE and the associated error is zero. Therefore, we can rewrite \eqref{eq:3.AB} as
\[
\|\bs e^h(t)\|_{1,\RdG}+\|\bs J\bs e^h(t)\|_{H_\Gamma}\le C \max\{1,t\} \mathrm{Acc}(\bs\psi-\bs\Pi_h\bs\psi,t,\theta),
\]
which proves \eqref{eq:3.AA}.
\end{proof}

\section{Multistep CQ time discretization}
\label{sect:multistep}

In this section we introduce and analyze Convolution Quadrature schemes, based on BDF time-integrators, applied to the semidiscrete TDBIE \eqref{eq:8} and to the potential postprocessing \eqref{eq:9}. All convolution operators --- in the left and right hand sides of \eqref{eq:8} and in the retarded potentials in \eqref{eq:9} --- will be treated with BDF-CQ.

\subsection{The algorithm}

Consider a constant $k>0$ and a sequence of discrete time steps $t_n := nk$ for $n\geq 0$ and let
\begin{equation}\label{eq:4.1}
\delta(\zeta ):=\sum_{j=1}^q {1\over j} (1-\zeta)^j =: \sum_{j=0}^q \alpha_j \zeta^j
\end{equation}
be the characteristic polynomial of the BDF($q$) method. If we write the $\zeta$-transform of a sequence of samples of a causal $X$-valued function $v$,
\[
V(\zeta ) := \sum_{n=0}^\infty v(t_n) \zeta^n,
\]
then
\begin{equation}\label{eq:4.2}
\frac1k \delta(\zeta) V(\zeta) 
=\frac1k \sum_{n=0}^\infty \left(\sum_{j=0}^{\min\{q,n\}}
\alpha_j v(t_{n-j})\right)\zeta^n
\approx \dot V(\zeta) = \sum_{n=0}^\infty \dot v(t_n) \zeta^n.
\end{equation}
Note that the approximation of the derivative is only good when $v$ is a smooth causal function.
In terms of the $\zeta$-transform of data
\[
B_0(\zeta) := \sum_{n=0}^\infty \beta_0(t_n) \zeta^n,\qquad
B_1(\zeta) := \sum_{n=0}^\infty \beta_1(t_n) \zeta^n,
\]
and of the fully discrete unknowns
\begin{alignat*}{3}
&\Lambda^{h,k}(\zeta) := \sum_{n=0}^\infty \lambda_n^{h,k} \zeta^n \approx \sum_{n=0}^\infty \lambda^h(t_n)\zeta^n, &&\quad
\Phi^{h,k}(\zeta) := \sum_{n=0}^\infty \phi_n^{h,k} \zeta^n \approx \phi^h(t_n) \zeta^n,\\
&U_-^{h,k}(\zeta) := \sum_{n=0}^\infty u_{-,n}^{h,k} \zeta^n
\approx \sum_{n=0}^\infty u^h_-(t_n)\zeta^n, && \quad
U_+^{h,k}(\zeta) := \sum_{n=0}^\infty u_{+,n}^{h,k} \zeta^n
\approx \sum_{n=0}^\infty u^h_+(t_n)\zeta^n,
\end{alignat*}
the fully discrete CQ-BEM equations look for $(\Lambda^{h,k}(\zeta),\Phi^{h,k}(\zeta))\in X_h\times Y_h$ such that
\begin{equation} \label{eq:45}
\begin{aligned}
\begin{bmatrix}
	 \VV(\frac{1}{k m} \delta (\zeta))+\kappa \VV(\frac1k \delta (\zeta)) & - \KK(\frac{1}{k m} \delta (\zeta))- \KK(\frac1k \delta (\zeta)) \\[0.05in]
	 \KK^T(\frac{1}{k m} \delta (\zeta))+ \KK^T(\frac1k \delta (\zeta))   &  \WW(\frac{1}{k m} \delta (\zeta))+\frac{1}{\kappa} \WW(\frac1k \delta(\zeta))
\end{bmatrix}
\begin{bmatrix} \Lambda^{h,k}(\zeta) \\[0.05in] \Phi^{h,k}(\zeta) \end{bmatrix} & \\
-
\frac12\begin{bmatrix} B_0(\zeta) \\[0.05in] \frac{1}{\kappa} B_1(\zeta) \end{bmatrix}
- \begin{bmatrix}
	 \VV(\frac1k \delta(\zeta)) & -\KK(\frac1k \delta(\zeta)) \\[0.05in]
	\frac{1}{\kappa} \KK^T(\frac1k \delta(\zeta)) & \frac{1}{\kappa} \WW(\frac1k \delta(\zeta))
\end{bmatrix}
\begin{bmatrix} B_1(\zeta) \\[0.05in] B_0(\zeta) \end{bmatrix} & \in X_h^\circ \times Y_h^\circ ,
\end{aligned}
\end{equation}
and then postprocess their output to build
\begin{subequations}\label{eq:46}
\begin{align}
U_-^{h,k}(\zeta) & = \SS(\tfrac{1}{k m} \delta(\zeta)) \Lambda^{h,k}(\zeta) - \DD(\tfrac{1}{k m} \delta(\zeta))\Phi^{h,k}(\zeta), \\
U_+^{h,k}(\zeta) &=-\SS(\tfrac1k \delta(\zeta))(\kappa \Lambda^{h,k}(\zeta)-B_1(\zeta))+ \DD(\tfrac1k \delta(\zeta)) (\Phi^{h,k}(\zeta)-B_0(\zeta)).
\end{align}
\end{subequations}
This short-hand exposition of the BDF-CQ method can be easily derived by taking the Laplace transform of \eqref{eq:8} and \eqref{eq:9} and substituting the Laplace transformed variable $s$ by the discrete symbol $k^{-1}\delta(\zeta)$. For readers who are not acquainted with CQ techniques, we explain in Appendix \ref{app:A2} the meaning of formulas \eqref{eq:45} and \eqref{eq:46}. 

\subsection{The analysis}

We can think of \eqref{eq:45}-\eqref{eq:46} as a `frequency-domain' system of BIE followed by potential postprocessing associated to a transmission problem with diffusion parameters $\rho \kappa^{-1}\,k^{-1}\delta(\zeta)$ and $k^{-1}\delta(\zeta)$. Therefore, if we write $\bs u_n^{h,k}:=(u_{-,n}^{h,k},u_{+,n}^{h,k})$, $\bs\chi(t):=((\beta_0(t),0,0,0),(0,0,0,\beta_1(t))$, and recall \eqref{eq:4.2}, it follows that
\begin{subequations}
\begin{alignat}{6}
\partial_k \bs u_n^{h,k} &=A_\star \bs u_n^{h,k},\\
\mathcal B\bs u_n^{h,k}
-\bs\chi(t_n) &\in \bs M,
\end{alignat}
\end{subequations}
where
\[
\partial_k \bs u_n^{h,k}=
\frac1k\sum_{j=0}^{\min\{q,n\}}\alpha_j \bs u_{n-j}^{h,k}
\]
is the backward derivative associated to the BDF scheme. On the other hand, the semidiscrete solution $\bs u^h$ satisfies very similar equations at the discrete times
\begin{subequations}
\begin{alignat}{6}
\dot{\bs u}^h(t_n) &=A_\star \bs u^h(t_n) ,\\
\mathcal B\bs u^h(t_n) 
-\bs\chi(t_n) &\in \bs M.
\end{alignat}
\end{subequations}
Therefore, the error $\bs e_n:=\bs u_n^{h,k}-\bs u^h(t_n)$ satisfies the equations
\begin{subequations}\label{eq:4.7}
\begin{alignat}{6}
\bs e_n &\in \bs D_h=D(A),\\
\partial_k \bs e_n &=A\bs e_n+\bs\theta_n,
\end{alignat}
\end{subequations}
where
$
\bs\theta_n:=\dot{\bs u}^h(t_n)-\partial_k \bs u^h(t_n)
$
is the error associated to the finite difference approximation of the time derivative. Note that
\begin{equation}\label{eq:4.888}
\bs J\bs e_n=\bs\psi_n^{h,k}-\bs\psi^h(t_n)
=(\phi_n^{h,k}-\phi^h(t_n),\lambda_n^{h,k}-\lambda^h(t_n)).
\end{equation}

\begin{theorem}
If $\bs u^h$ is smooth enough, then for all $n\ge 0$
\begin{subequations}
\begin{alignat}{6}
\label{eq:4.88a}
\| \bs u_n^{h,k}-\bs u^h(t_n)\|_\RdG & \le
C t_n k^q \max_{t\le t_n}\left\| \tfrac{\mathrm d^q}{\mathrm d t^q}\bs u^{h}(t)\right\|_\RdG,\\
\label{eq:4.88b}
\| \nabla \bs u_n^{h,k}-\nabla \bs u^h(t_n)\|_\RdG 
+\|\bs\psi_n^{h,k}-\bs\psi^h(t_n)\|_{H_\Gamma}
& \le C t_n k^q \max_{t\le t_n}\left\| \tfrac{\mathrm d^{q+1}}{\mathrm d t^{q+1}}\bs u^{h}(t)\right\|_\RdG.
\end{alignat}
\end{subequations}
\end{theorem}

\begin{proof}
Following \cite[Lemma 10.3]{Thomee2006}, we can show that the solution of the recurrence \eqref{eq:4.7} is
\begin{equation}\label{eq:4.9}
\bs e_n=k \sum_{j=0}^{n-1} P_j(-k\,A) (\alpha_0I-kA)^{-j-1} \bs\theta_{n-j}
\end{equation}
(here $\alpha_0=\delta(0)$ is the leading coefficient of the BDF derivative), where  $\{ P_j\}$ is a sequence of polynomials  with $\mathrm{deg}\,P_j\le j$ and
\begin{equation}\label{eq:4.8}
\sup_{z>0}\left| \frac{P_j(z)}{(\alpha_0+z)^{j+1}}\right| \le C \qquad \forall j.
\end{equation}
The proof of \eqref{eq:4.9} is purely algebraic, based only on the fact that $ \alpha_0I-kA $ can be inverted.
The rational functions in \eqref{eq:4.8} are bounded at infinity and have all their poles at $-\alpha_0<0$, which allows us to apply Theorem \ref{the:A2} and show that
\[
\|\bs e_n\|_\RdG \le C k \sum_{j=1}^{n} \|\bs\theta_{j}\|_\RdG 
\le C t_n \max_{1\le j\le n} \|\bs\theta_j\|_\RdG.
\]
However, if $\bs u^h$ is smooth enough (as a function from $\mathbb R$ to $\bs H=L^2(\RdG)^2)$,  Taylor expansion yields
\[
\|\bs\theta_j\|_\RdG\le C k^q \max_{t_{j-q}\le t\le t_j}\left\| \tfrac{\mathrm d^q}{\mathrm d t^q}\bs u^{h}(t)\right\|_\RdG,
\]
which proves \eqref{eq:4.88a}.

Note now that 
\[
\partial_k \bs\theta_n=\partial_k \dot{\bs u}^h(t_n)-\ddot{\bs u}^h(t_n)
-(\partial_k^2\bs u^h(t_n)-\ddot{\bs u}^h(t_n))
\]
and that $\bs f_n:=\partial_k\bs e_n\in \bs D_h$ satisfies the recurrence
$
\partial_k \bs f_n =A\bs f_n+\partial_k\bs\theta_n.
$
Using a simple argument on Taylor expansions and Theorem \ref{the:A2}, it follows that
\[
\|\bs e_n\|_\RdG \le C t_n \max_{1\le j\le n} \|\partial_k\bs\theta_j\|_\RdG
\le C t_n \max_{t\le t_n}\left\| \tfrac{\mathrm d^{q+1}}{\mathrm d t^{q+1}}\bs u^{h}(t)\right\|_\RdG.
\]
This can be used to give a bound for $A_\star\bs e_n=\partial_k\bs e_n-\bs\theta_n$ and \eqref{eq:NEW24} then proves that
\[
\|\bs e_n\|_{1,\RdG}+\|\Delta\bs e_n\|_\RdG
\le C t_n \max_{t\le t_n}\left\| \tfrac{\mathrm d^{q+1}}{\mathrm d t^{q+1}}\bs u^{h}(t)\right\|_\RdG.
\]
Using this bound, \eqref{eq:4.888}, and the boundedness of $\bs J$, \eqref{eq:4.88b} follows.
\end{proof}

\section{Multistage CQ time discretization} \label{sec:6}

In this section we introduce and analyze some Runge-Kutta based Convolution Quadrature (RKCQ) schemes for the full discretization of the semidiscrete system of TDBIE \eqref{eq:8} and the potential postprocessing \eqref{eq:9}. Some background material and references on RKCQ and the needed Dunford calculus can be found in Appendices \ref{app:B} and \ref{app:A4}.

\subsection{The algorithm and some observations}

We consider an implicit $s$-stage RK method with Butcher tableau
\[
\begin{array}{c|c}
\bff c & \mathcal{Q}\\ \hline
& \\[-2ex]
 & \bff b^T
\end{array} \qquad
\begin{array}{c|cccc}
c_1 & a_{11} & a_{12} & \cdots & a_{1s} \\
c_2 & a_{21} & a_{22} & \cdots & a_{2s} \\
\vdots & \vdots & \vdots & \ddots & \vdots \\
c_s & a_{s1} & a_{s2} & \cdots & a_{ss} \\ \hline
 & b_1 & b_2 & \cdots & b_s
\end{array}
\]
and stability function
\[
r(z) := 1+z\bff b^T(\mathcal I-z\mathcal Q)^{-1}\bff 1,
\]
where $\bff 1={(1,\ldots ,1)}^{T} \in \mathbb R^s$ and $\mathcal I$ is the $s\times s$ identity matrix.  We will assume the following hypotheses on the RK method:
\begin{itemize}
\item[(a)] The method has (classical) order $p$ and stage order $q\leq p-1$. We exclude methods where $q=p$ for simplicity. (For instance, the one-stage backward Euler formula, which was covered as the BDF$(1)$ method in the previous section, is not included in this exposition.)
\item[(b)] The method is A-stable, i.e., the matrix $\mathcal I-z\mathcal Q$ is invertible for $\operatorname{Re} z\leq 0$
  and the stability function satisfies $|r(z)|\leq 1$  for those values of $z$.
\item[(c)] The method is stiffly accurate, i.e., 
$\bff b^T \mathcal{Q}^{-1}  = (0,0,\ldots ,0,1)=:\mathbf e_s^T.$
This implies $\lim_{|z|\to\infty} |r(z)|=0$.  

Assuming the usual simplifying hypothesis for RK schemes $\mathcal Q\mathbf 1=\mathbf c$, stiff accuracy implies that $c_s=1$, that is, the last stage of the method is the step. Stiff accuracy also implies that (cf. \cite[Lemma 2]{BaLuMe2011})
\begin{alignat*}{6}
r(z)
= & 
\mathbf b^T\mathcal Q^{-1}\mathbf 1 + z\mathbf b^T (\mathcal I-z\mathcal Q)^{-1}\mathbf 1 \\
=& \mathbf b^T \mathcal Q^{-1} (\mathcal I-z\mathcal Q+ z\mathcal Q)
(\mathcal I-z\mathcal Q)^{-1}\mathbf 1 \\
=& \bff b^T\mathcal Q^{-1} (\mathcal I - z\mathcal Q)^{-1} \bff 1.
\end{alignat*}

\item[(d)] The matrix $\mathcal Q$ is invertible. This hypothesis and A-stability imply that the spectrum of $\mathcal Q$ is contained in $\mathbb C_+$. 
\end{itemize}
Examples of the Runge-Kutta methods satisfying all the hypotheses above are provided by the family of $s$-stage Radau IIA methods with order $p=2s-1$ and stage order $q=s$. 

Given a function of the time variable, we will write
\[
w(t_n+\bff ck) := \begin{bmatrix} w(t_n+c_1k) \\ w(t_n+c_2k) \\ \vdots \\ w(t_n+c_s k)\end{bmatrix}
\]
to denote the $s$-vectors with the samples at the stages in the time interval $[t_n,t_{n+1}]$. We sample the boundary data and collect the vectors of time samples in formal $\zeta$-series
\[
B_0(\zeta) := \sum_{n=0}^\infty \beta_0(t_n+\bff ck) \zeta^n,\qquad
B_1(\zeta) := \sum_{n=0}^\infty \beta_1(t_n+\bff ck) \zeta^n.
\]
The unknowns for the fully discrete method can be collected in 
\begin{alignat*}{3}
&\Lambda^{h,k}(\zeta) := \sum_{n=0}^\infty \lambda^{h,k}_n \zeta^n 
\approx \sum_{n=0}^\infty \lambda^h(t_n+\bff ck)\zeta^n, &&\quad
\Phi^{h,k} (\zeta):= \sum_{n=0}^\infty \phi^{h,k}_n \zeta^n \approx \sum_{n=0}^\infty\phi^h(t_n+\bff ck) \zeta^n,\\
&U_-^{h,k} (\zeta):= \sum_{n=0}^\infty U_{-,n}^{h,k} \zeta^n
\approx \sum_{n=0}^\infty u^h_-(t_n+\bff ck)\zeta^n, && \quad
U_+^{h,k} (\zeta):= \sum_{n=0}^\infty U_{+,n}^{h,k} \zeta^n
\approx \sum_{n=0}^\infty u^h_+(t_n+\bff ck)\zeta^n,
\end{alignat*}
The pairs $(\lambda^{h,k}_n,\phi^{h,k}_n)\in X_h^s\times Y_h^s$ are computed using \eqref{eq:45} where the symbol $\delta$ is now the matrix-valued RK differentiation operator
\begin{equation}\label{eq:deltazeta}
\delta (\zeta) = 
\left(\mathcal Q+\frac\zeta{1-\zeta}\mathbf 1\mathbf b^T\right)^{-1}=\mathcal{Q}^{-1} - \zeta \mathcal{Q}^{-1} \bff 1\bff{b}^T \mathcal{Q}^{-1},
\end{equation}
(these two matrices can be easily seen to be equal using the stiff accuracy hypothesis) and the testing condition has to be modified, imposing that the residual is in $(X_h^s\times Y_h^s)^\circ\equiv (X_h^\circ)^s\times (Y_h^\circ)^s$. The discrete potentials at the different stages can be computed using \eqref{eq:46}, with the new definition of $\delta(\zeta)$. In all the expressions for the RK-CQ fully discrete equations, analytic functions are evaluated at $k^{-1}\delta(\zeta)$ via Dunford calculus (see Appendix \ref{app:B}). This is meaningful since the spectrum of $\delta(\zeta)$ lies in $\mathbb C_+$ for $\zeta$ small enough, which is due to the fact that $\delta(\zeta)$ is a small (rank-one) perturbation of $\mathcal Q^{-1}$ and the spectrum of $\mathcal Q$ is in $\mathbb C_+$ (see hypotheses (b) and (d) above). 

Before we embark ourselves in the error analysis of the fully discrete method, which involves using quite non-trivial results from \cite{AlPa2003}, we are going to make some important remarks that will be pertinent to the analysis.
\begin{itemize}
\item[(1)] Using classical results on interpolation spaces on (bounded and unbounded) Lipschitz domains and the identification of Sobolev spaces with or without Dirichlet condition for low order indices (see \cite[Theorem 3.33, Theorem B.9, Theorem 3.40]{McLean2000}) it follows that for all $\mu<1/2$ 
\begin{alignat*}{6}
H^1(\mathbb R^d\setminus\Gamma)\equiv H^1(\Omega_-)\times H^1(\Omega_+)
 &\subset H^\mu(\Omega_-)\times H^\mu(\Omega_+) 
  =H^\mu_0(\Omega_-)\times H^\mu_0(\Omega_+) \\
 & = [L^2(\Omega_-),H^2_0(\Omega_-)]_{\mu/2} \times
 	 [L^2(\Omega_+),H^2_0(\Omega_+)]_{\mu/2}\\
& \equiv [L^2(\mathbb R^d),H^2_0(\mathbb R^d\setminus\Gamma)]_{\mu/2},
\end{alignat*}
and therefore
\[
\bs V=H^1(\mathbb R^d\setminus\Gamma)^2
\subset [\bs H,H^2_0(\mathbb R^d\setminus\Gamma)^2]_\nu
\subset [\bs H,D(A)]_\nu=D((I-A)^\nu) \qquad \forall \nu <1/4.
\]
\item[(2)] Because of the hypotheses on the RK method, the rational function $R(z):=r(-z)$ satisfies the conditions of Theorem \ref{the:A2} (it is bounded at infinity and has all its poles in the negative real part complex half-plane). We also know that the operator $-A$ is self-adjoint and non-negative (Proposition \ref{prop:4}(c)). Therefore,
\begin{equation}\label{eq:920a}
\| r(kA)\|_{\bs H\to\bs H}=\| R(-kA)\|_{\bs H\to \bs H}
\le \sup_{z>0}|R(z)|=\sup_{z<0}|r(z)|= 1,
\end{equation}
where we have used $A$-stability of the RK scheme.
Consequently
\begin{equation}\label{eq:920b}
\sup_n \| r(kA)^n\|_{\bs H \to \bs H}\le 1 \qquad \forall n,\quad \forall k>0.
\end{equation}
\item[(3)] The entries of the matrix-valued rational function $z\mapsto (\mathcal I+z\mathcal Q)^{-1}$ are rational functions with poles in $\{ z\,:\,\mathrm{Re}\,z<0\}$ and converging to zero as $|z|\to \infty$. Using  Theorem \ref{the:A2} in a similar way to how we proved \eqref{eq:920a} above, it follows that
\begin{equation}\label{eq:920c}
\| (\mathcal I\otimes I-k\mathcal Q \otimes A)^{-1}\|_{\bs H^s \to \bs H^s}\le C \qquad \forall k>0.
\end{equation}
Here we have used traditional tensor product notation for Kronecker products of $s\times s$ matrices by operators. 
\end{itemize}

\subsection{Error estimates}

The following proposition basically says that using RKCQ in the semidiscrete system of integral equations and then in the potential representation is equivalent to applying RK to the transmission problem associated to the heat equation satisfied by the potential fields. 

\begin{proposition}\label{prop:55.1}
If $\bs U^{h,k}_n:=(U^{h,k}_{-,n},U^{h,k}_{+,n})\in \bs D^s$ is the vector of the internal stages of the RKCQ method in $[t_n,t_{n+1}]$, then
\begin{subequations}\label{eq:555.2}
\begin{alignat}{6}
\bs U^{h,k}_n-k(\mathcal Q\otimes A_\star) \bs U^{h,k}_n 
	&=\mathbf 1\mathbf e_s^T \bs U^{h,k}_{n-1},\\
\mathcal B \otimes \bs U^{h,k}_n -\bs\chi(t_n+\mathbf c\,k) &\in \bs M^s.
\end{alignat}
\end{subequations}
\end{proposition}

\begin{proof}
If we collect data and potential fields in 
\begin{alignat*}{6}
U^{h,k}(\zeta)& :=(U^{h,k}_-(\zeta),U^{h,k}_+(\zeta))\in \bs D^s\\
\Xi(\zeta)&:=((B_0(\zeta),0,0,0),(0,0,0,B_1(\zeta))\in \bs H_\Gamma^s,
\end{alignat*}
it then follows that
\begin{subequations}
\begin{alignat}{6}
\label{eq:555.3a}
 k^{-1}\delta(\zeta)   U^{h,k}(\zeta) &=(\mathcal I \otimes A_\star) U^{h,k}(\zeta),\\
 \label{eq:555.3b}
\mathcal B\otimes U^{h,k}(\zeta)-\Xi(\zeta) &\in \bs M^s.
\end{alignat}
\end{subequations}
(See the Appendix of \cite{MeRi2017} for a detailed argument showing why this holds in the context of wave equations. Those ideas can be used almost verbatim for the heat equation.) Equation \eqref{eq:555.3a} is equivalent to
\begin{equation}\label{eq:555.4}
U^{h,k}(\zeta)-\zeta \mathbf 1\mathbf e_s^T 
U^{h,k}(\zeta)
=k(\mathcal Q\otimes A_\star) U^{h,k}(\zeta). 
\end{equation}
Looking at the time instances of the discrete-in-time equations enconded in the $\zeta$-transformed equations \eqref{eq:555.4} and \eqref{eq:555.3b}, the result follows. 
\end{proof}

\begin{proposition}[$L^2$ error estimate for the steps]\label{prop:5.2}
Let $\bs u^{h,k}_n=\bs e_s^T\bs U^{h,k}_n$ be the $n$-th step approximation provided by the RKCQ method. We have the bounds: if $p=q+1$, then
\[
\|\bs u^{h,k}_n-\bs u^h(t_n)\|_\Rd
\le C  k^{q+1}
c(\bs u^h,t_n),
\]
whereas if $q\leq p+2$ and $\varepsilon\in (0,1/4]$,
\[
\|\bs u^{h,k}_n-\bs u^h(t_n)\|_\Rd
\le C_\varepsilon  k^{q+1+\frac14-\varepsilon}
c(\bs u^h,t_n).
\]
Here,
\begin{equation}
c(\bs u^h,t_n):=(1+ t_n)\sum_{\ell=q+1}^{p+1}\max_{  t\le t_n} 
\left\| \tfrac{\mathrm d^\ell}{\mathrm d t^\ell}\bs u^h(t)\right\|_{1,\RdG}.
\end{equation}
\end{proposition}

\begin{proof}
If $\mathcal P: \bs H_\Gamma \to \bs M^\bot$ is the orthogonal projection onto $\bs M^\bot$ (the orthogonality is with respect to the $\bs H_\Gamma$ inner product), then the semidiscrete equations (see \eqref{eq:AA18}) are equivalent to
\begin{subequations}\label{eq:555.5}
\begin{alignat}{6}
\dot{\bs u}^h(t) &=A_\star \bs u^h(t),\\
\mathcal P\mathcal B \bs u^h(t) & = \mathcal P \bs\chi(t),\\
\bs u^h(0) &=0.
\end{alignat}
\end{subequations}
(Note that the last condition is equivalent to $\mathcal B\bs u(t)-\bs\chi(t)\in \bs M$.) If we apply the RK method to equations \eqref{eq:555.5}, and we recall that the method is stiffly accurate, we obtain that the computation of the internal stages is given by
\begin{subequations}\label{eq:555.6}
\begin{alignat}{6}
\bs U^{h,k}_n
	&=\mathbf 1\mathbf e_s^T \bs U^{h,k}_{n-1}
		+k(\mathcal Q\otimes A_\star) \bs U^{h,k}_n ,\\
\mathcal P\mathcal B \otimes \bs U^{h,k}_n &=\mathcal P\otimes \bs\chi(t_n+\mathbf c\,k).
\end{alignat}
\end{subequations}
These equations are clearly equivalent to equations \eqref{eq:555.2}, which have been shown to be equivalent to the equations satisfied by the fields obtained in the RKCQ method. In summary, we are dealing here with the direct application of the RK method to equations \eqref{eq:555.5}.

This result is now a consequence of one of the main theorems of \cite{AlPa2003}. Unfortunately the reader will be now teleported from the middle of this proof to the core of a highly technical article. We will just give a translation guide to help with the application of the results of that paper. In our context, and taking advantage of our limitation to methods with $p\ge q+1$, the key theorem of \cite{AlPa2003} is Theorem 2. This result is stated for problems with homogeneous boundary conditions, but Section 4 of \cite{AlPa2003} explains how to handle the non-homogeneous boundary conditions and why the result still holds. The following translation table 
\[
\begin{array}{r||c|c|c|c|c|c|c|c|c|c|c|}
\mbox{\cite{AlPa2003}} 
& \widetilde A & \partial & D(\widetilde A) & D(A) & S_\alpha & Z_\alpha & \rho_k(T) 
& m & \nu & \nu^* & \theta\\ 
\hline 
& & & & & & & & & & &  \\[-2ex]
\mbox{Here} 
& A_\star & \mathcal P\mathcal B & \bs D & \bs D_h & (-\infty,0] & \emptyset & 1 
& 0 & \min\{p-q-1,\frac14-\varepsilon\} & 1 & 0 
\end{array}
\]
can be used to navigate \cite[Theorem 2]{AlPa2003} and relate it to our particular problem. Note that when $p\ge q+2$, it is important to have $\bs D\subset \bs V \subset [\bs H,D(A)]_{1/4-\varepsilon}$ with bounded embeddings. This is a hypothesis needed in the application of the theorems and allows us to eliminate the estimates in terms of interpolated norms and write them in the natural Sobolev norm of $\bs V$. 
\end{proof}

We note that the error for the case $p=q+1$ can be written with the $\bs H=L^2(\RdG)^2$ norm in the right-hand-side. We will keep the $\bs V= H^1(\RdG)^2$ overestimate in this case for simplicity. 

\begin{proposition}[$L^2$ error for the internal stages]\label{prop:5.3}
With $c(\bs u^h,t_n)$ defined as in Proposition \ref{prop:5.2}, we have the bounds
\[
\| \bs U^{h,k}_n-\bs u^h(t_n+\mathbf c\,k)\|_{\Rd}
\le C k^{q+1} 
c(\bs u^h,t_n).
\]
\end{proposition}

\begin{proof}
Let $\bs E_n:=\bs U^{h,k}_n-\bs u^h(t_n+\mathbf c\,k)\in \bs D_h^s=D(A)^s$ and note that
\begin{equation}\label{eq:921a}
\bs E_n-k(\mathcal Q\otimes A) \bs E_n
=\mathbf 1\mathbf e_s^T \bs E_{n-1}+k \mathcal Q \bs\Theta_n,
\end{equation}
where 
\begin{alignat*}{6}
\bs \Theta_n:= &  \dot{\bs u}^h(t_n+\mathbf c\,k)
- k^{-1} \mathcal Q^{-1} (\bs u^h(t_n+\mathbf c\,k)-
\mathbf 1\mathbf e_s^T\bs u^h(t_{n-1}+\mathbf c\,k)) \\
= & \dot{\bs u}^h(t_n+\mathbf c\,k)
- k^{-1} \mathcal Q^{-1} (\bs u^h(t_n+\mathbf c\,k)-
\mathbf 1 u^h(t_n)).
\end{alignat*}
Using Taylor expansions and the fact that a method with stage order $q$ satisfies $\ell \mathcal Q\mathbf c^{\ell-1}=\mathbf c^\ell$ for $1\le \ell\le q$ (powers of a vector are taken componentwise), we can easily prove that
\begin{equation}\label{eq:920d}
\|\bs\Theta_n\|_\Rd\le C k^q \max_{t_n\le t\le t_{n+1}} 
\left\| \tfrac{\mathrm d^{q+1}}{\mathrm dt^{q+1}}\bs u^h(t)\right\|_\Rd.
\end{equation}
Therefore, by \eqref{eq:920c}, we can bound
\begin{alignat*}{6}
\|\bs E_n\|_{\bs H_s} 
	\le & C ( \|\bs u^{h,k}_{n-1}-\bs u^h(t_{n-1})\|_\Rd+k \|\bs\Theta_n\|_\Rd)
\end{alignat*}
and the result follows by applying \eqref{eq:920d} and Proposition \ref{prop:5.2}.
\end{proof}

\begin{proposition}[$H^1$ error and estimates for boundary unknowns]
\label{prop:5.4}
With the definition of $c(\bs u^h,t_n)$ given in Proposition \ref{prop:5.2},
we have the estimates
\begin{alignat*}{6}
\|\phi^{h,k}_n-\phi^h(t_n+\mathbf c k)\|_{1/2,\Gamma}
+\| \bs U^{h,k}_n-\bs u^h(t_n+\mathbf c\, k)\|_{1,\RdG}
\le & C  k^{q+\frac12} c(\bs u^h,t_n),\\
\|\lambda^{h,k}_n-\lambda^h(t_n+\mathbf c\,k)\|_{-1/2,\Gamma}
\le & C k^q c(\bs u^h,t_n).
\end{alignat*}
\end{proposition}

\begin{proof}
From \eqref{eq:921a}, we have
\[
A\otimes \bs E_n=\mathcal Q^{-1}(\mathcal Q\otimes A)\bs E_n=
k^{-1}(\bs E_n-\mathbf 1\mathbf e_s^T\bs E_{n-1})-\bs\Theta_n,
\]
and therefore
\[
\| A_\star \otimes \bs U^{h,k}_n-A_\star\otimes \bs u^h(t_n+\mathbf c\,k)\|_\RdG
=\|A\otimes \bs E_n\|_\RdG
\le C k^q c(\bs u^h,t_n),
\]
by Proposition \ref{prop:5.3} and \eqref{eq:920d}. Therefore, by \eqref{eq:BB22}
\[
|\bs U^{h,k}_n-\bs u^h(t_n+\mathbf c\,k)|_{\bs V^s}
=|\bs E_n|_{\bs V^s}
\le \| \bs E_n\|_{\bs H^s}^{1/2}\| A\otimes \bs E_n\|_{\bs H^s}^{1/2}
\le C k^{q+\frac12} c(\bs u^h,t_n).
\]
The result is therefore clear, since the errors in the boundary quantities can be derived from jump relations applied to $\bs E_n$ and the $H^1(\RdG)^2$ norm can be bounded by the sum of the $\bs V$ seminorm and the $\bs H$ norm. 
\end{proof}

\section{Numerical experiments}
\label{sect:numerics}
In order to confirm our theoretical findings, we conduct some numerical experiments in $\mathbb{R}^2$.  
For this we combine a  frequency-domain Galerkin-BEM code with the fast CQ-algorithm from~\cite{banjai_sauter_rapid_wave}. Note that a faster implementation can be achieved using the Fast and Oblivious CQ method \cite{ScLFLu2006}. 
For the discretization in space we use piecewise polynomial spaces; for $X_h$ we use discontinuous polynomials of degree $p$ and for $Y_h$ we use globally
continuous piecewise polynomials of degree $p+1$. We denote these pairings as  $\mathcal {P}_p-\mathcal{P}_{p+1}$.

\paragraph{Testing on a manufactured solution.}
We start our experiments with the case where the exact solution can be  computed analytically. This allows us to test the
predicted convergence rates from Sections~\ref{sect:multistep} and~\ref{sec:6}.

Since we are mainly interested in the performance of the CQ-schemes, i.e., the discretization in time we try to use a spatial discretization of
higher order than the time discretization, while keeping the ratio $k/h$ of timestep size and mesh-width constant. This means,
whenever we cut the timestep size in half, we perform a uniform refinement of the spatial grid. See Table~\ref{tab:expected_rates}
for the degrees used.
We note that for smooth solutions, we expect convergence rates in space  of order $\bigO(k^{p+1})$ for the 
quantity $\|\lambda- \lambda^h\|_{L^2(\Gamma)} + \|\phi -\phi^h\|_{H^{1/2}(\Gamma)}$ when using the space $\mathcal{P}_p - \mathcal{P}_{p+1}$
(see e.g.~\cite{SaSc2011}).

The domain $\Omega_-$ is the quadrialteral with vertices $(0,0),(1,0),(0.8,0.8), (0.2,1).$
The thermal transmission constants are chosen to be $m:=\rho^{-1}\kappa=0.8, \kappa:=1.2$. We can then prescribe a  solution by
picking a source point outside the polygon $\bff x^{\mathrm{sc}}=(1.5,1.6)$ and defining:
\begin{align}
  \label{eq:point_source}
  u_-(\bff x,t)={1\over 4\pi m t}\exp \left( -{|\bff x-\bff x^{\mathrm{sc}}|^2\over 4m t} \right) \qquad
  u_+(\bff x,t)=0.
\end{align}
For our computations, we solve the system up to the fixed end-time $T=4$.

When using a Runge-Kutta method, we expect a reduction of order phenomenon, which depends on the norm under consideration,
see Proposition~\ref{prop:5.4}. Since they are easier to compute, we would like to consider the $L^2$-errors for the Dirichlet and Neumann traces.
For the Dirichlet-trace $\phi$, the $L^2(\Gamma)$-norm is weaker than the $H^{1/2}_\Gamma$-norm
in Proposition~\ref{prop:5.4}. We therefore expect a slightly higher rate of convergence. Namely, if we use the multiplicative 
trace estimate (see e.g. \cite[est. (1.6.2)]{brenner_scott}) we get:
\begin{align*}
  \|{\phi(t_n) - \phi^{h,k}(t_n)}\|_{\Gamma}&\lesssim \|\bs u(t_n) - \bs u^{h,k}(t_n)\|_{0,\mathbb{R}^d}^{1/2}
   \|\bs u(t_n) - \bs u^{h,k}(t_n)\|_{1,\mathbb{R}^d \setminus \Gamma}^{1/2}.
\end{align*}
Considering the error quantity
\begin{align*}
  E^{\phi}:=\max_{\max_{j \Delta t \leq T}}{\|{\phi(t_n) - \phi^{h,k}(t_n)}\|_{\Gamma}}
\end{align*}
we therefore expect a rate of $p_{e,\phi}:=\min\left\{q+3/4,\frac{3}{4}p+\frac{q}{4}\right\}$ (up to arbitrary $\varepsilon > 0$)
for the convergence in time.

For the normal derivative, the $L^2$ norm is stronger than what is covered by our theory.
Therefore, we also include an estimation for the true $H^{-1/2}_\Gamma$-norm. We thus consider
the following two quantities:
\begin{align*}
  E^{\lambda,0}&:=\max_{j \Delta t \leq T}\|\lambda(t_j) - \lambda^{h,k}(t_j)\|_{\Gamma}, \\
  E^{\lambda,-\frac{1}{2}}&:=\max_{j \Delta t \leq T}\|\Pi_{L^2}\lambda(t_j) - \lambda^{h,k}(t_j)\|_{-1/2,\Gamma},
\end{align*}
where $\Pi_{L^2}$ denotes the $L^2$-projection onto the boundary element space $X_h$. Since the exact solution is smooth
and the space-discretization error is taken to be of higher order than the time discretization, this should give a good
estimate for the $H^{-1/2}$-error. We predict rates of $p_{e,\lambda}:=q$. 

For multistep methods there is no reduction of order phenomenon, and we expect to see the full convergence order.
We collect the expected rates in Table~\ref{tab:expected_rates}.

Compared to the experiments, our estimates in Proposition~\ref{prop:5.4} do not appear to be sharp. It should be possible to improve the convergence estimate
  for $\|{\mathrm{ A} ( \bs u^{h}(t_n) - \bs u^{h,k}(t_n))\|}_{0,\mathbb{R}^d \setminus \Gamma}$ to $\bigO(k^{q+1/4-\varepsilon})$.
  Proving this estimate is more involved and part of future work. Using such sharper estimates, it should be possible to show an improved rate of
\[
  \widetilde{p}_{e,\phi}:=\min\left\{q+1,\frac{3}{4}p+\frac{q}{4}+\frac{1}{16}\right\}
  \qquad
  \widetilde{p}_{e,\lambda}:=q+1/4,
\]
for the $L^2$-error of $\phi$ and the $H^{-1/2}_\Gamma$-error of $\lambda$ respectively. We include
these rates as conjectured in Table~\ref{tab:expected_rates}.

\begin{table}
\center
\begin{tabular}{l|c|c|c|c|c}
  Method & $p_{e,\phi}$ &$p_{e,\lambda}$ & $\widetilde{p}_{e,\phi}$  (conjectured) & $\widetilde{p}_{e,\lambda}$ (conjectured) & $X_h\times Y_h$ \\
  \hline
  Radau IIa $s=2$ & 2.75 & 2 & 2.8125 & 2.25 & $\mathcal{P}_{3} - \mathcal{P}_4$\\
  Radau IIa $s=3$ & 3.75 & 3 & 4 & 3.25 & $\mathcal{P}_{4} - \mathcal{P}_5$\\
  BDF$(2)$ & 2 & 2 & -  & - & $\mathcal{P}_{2} - \mathcal{P}_3$\\
  BDF$(4)$ & 4 & 4 & -  & - & $\mathcal{P}_{4} - \mathcal{P}_5$
\end{tabular}
\caption{Expected convergence rates and space discretization used for the different convolution quadrature methods}
\label{tab:expected_rates}
\end{table}

\begin{figure}
\begin{subfigure}{0.45\textwidth}
  \includetikz{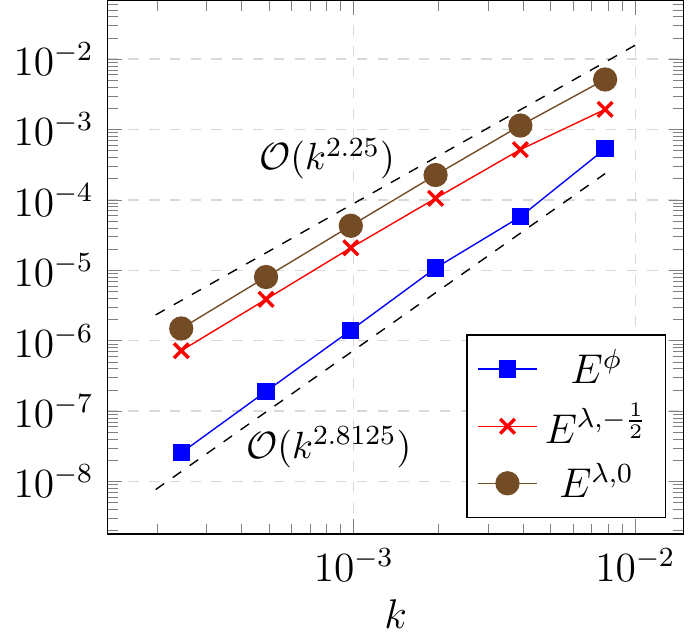}
  \caption{Convergence for 2-stage RadauIIa}
\end{subfigure}
\hfill
\begin{subfigure}{0.45\textwidth}
  \includetikz{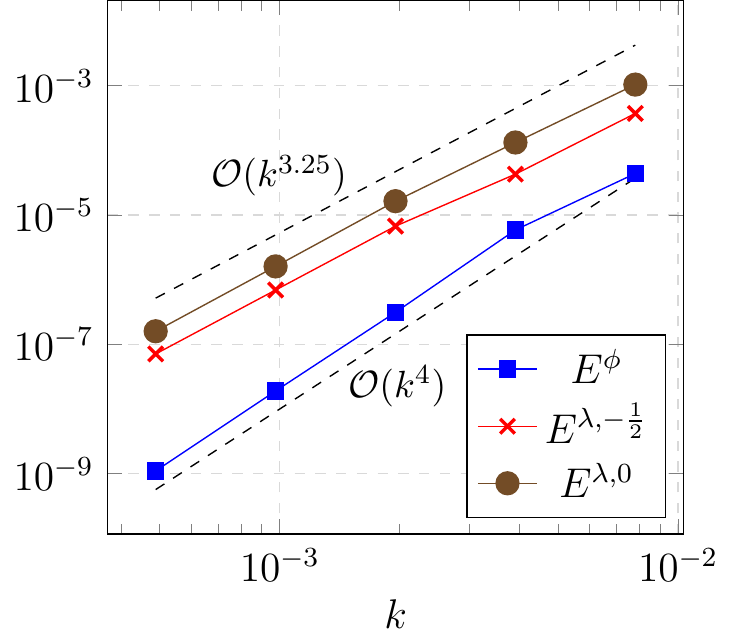}
  \caption{Convergence for 3-stage RadauIIa}
\end{subfigure}
\begin{subfigure}{0.45\textwidth}
  \vspace{5mm}
  \includetikz{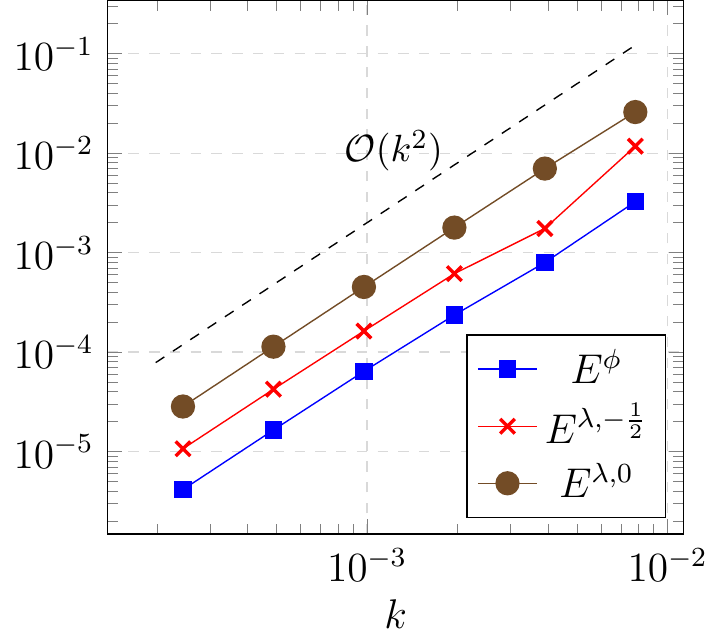}
  \caption{Convergence for BDF$(2)$}
\end{subfigure}
\hfill
\begin{subfigure}{0.45\textwidth}
  \vspace{5mm}
  \includetikz{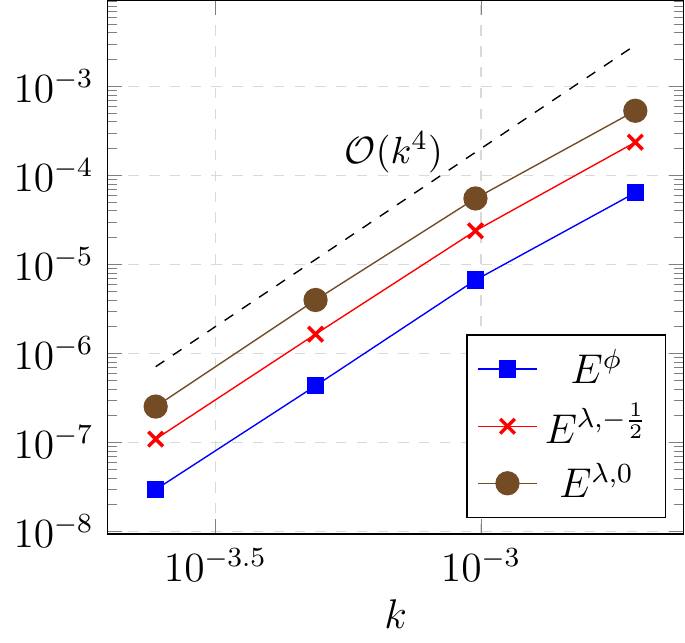}
  \caption{Convergence for BDF$(4)$}
\end{subfigure}

\caption{Comparison of different CQ methods }
\label{fig:conv_rk_ex1}
\end{figure}

In Figure~\ref{fig:conv_rk_ex1} we observe that the Runge-Kutta based methods slightly outperform the expected rates $p_{e,\phi}$ and $p_{e,\lambda}$.
Instead we get very good correspondence to the conjectured rates $\widetilde{p}_{e,\phi}$ and $\widetilde{p}_{e,\lambda}$. 
For multistep methods we see the full convergence rate, as was predicted in Section~\ref{sect:multistep}.

\paragraph{A simulation.}
We finally illustrate the use of our method for a simulation without known analytic solution.
We choose the domain $\Omega_-$ as a horseshoe shaped polygon, with a high conductivity 
parameter $\kappa:=100$. The density $\rho$ was set to $1$. We then placed point-sources of the form
\begin{align}
  u^{\text{src}}(\bff x,t)
:=\begin{cases} \displaystyle
  \sum_{j=0}^{m}{{1\over 4\pi t}\exp \left( -{|\bff x-\bff x^{\mathrm{sc}}_j|^2\over 4t} \right) } &\text{for $\bff x \in \Omega_+$}, \\
  0 & \text{ in $\Omega_-$},
\end{cases}
\end{align}
on points $\bff x^{\mathrm{sc}}_j$ uniformly distributed on a circle. Making the ansatz for the solution $u^{\text{tot}}=u^{\text{src}} + u$,
we can compute $u$ using our numerical method and recover $u^{\text{tot}}$ by postprocessing.
See Figure~\ref{fig:simulation:t0} for the geometric setting and initial condition.
(Note: in order to avoid the singularity at $t=0$, we shifted the functions by a small time $t_{\text{lag}}:=0.001$).

We then solved the evolution problem up to the final time $T=1$. We used the BDF$(4)$ method with a step-size of $k:=1/2048$. We used
$\mathcal{P}_3 - \mathcal{P}_4$ elements with $h \approx 1/64$.
The results can be seen in Figure~\ref{fig:simulation}.

\begin{figure}
  \def\mywidth{0.3\textwidth}
  \def\picwidth{5.5cm}
  \begin{subfigure}{\mywidth}
    \center
    \includegraphics[width=\picwidth]{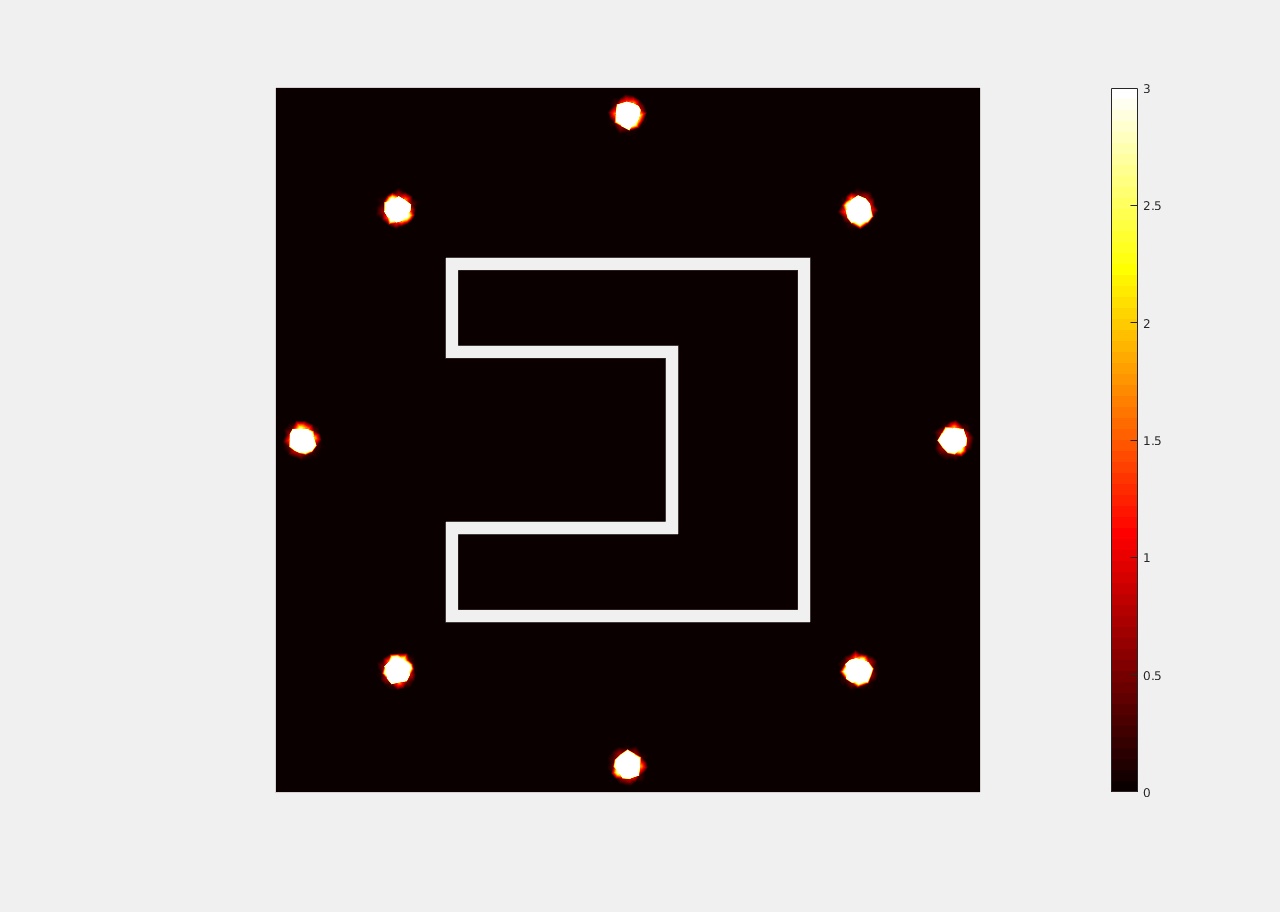}
    \caption{$t=0$}
    \label{fig:simulation:t0}
  \end{subfigure}
  \begin{subfigure}{\mywidth}
    \center
    \includegraphics[width=\picwidth]{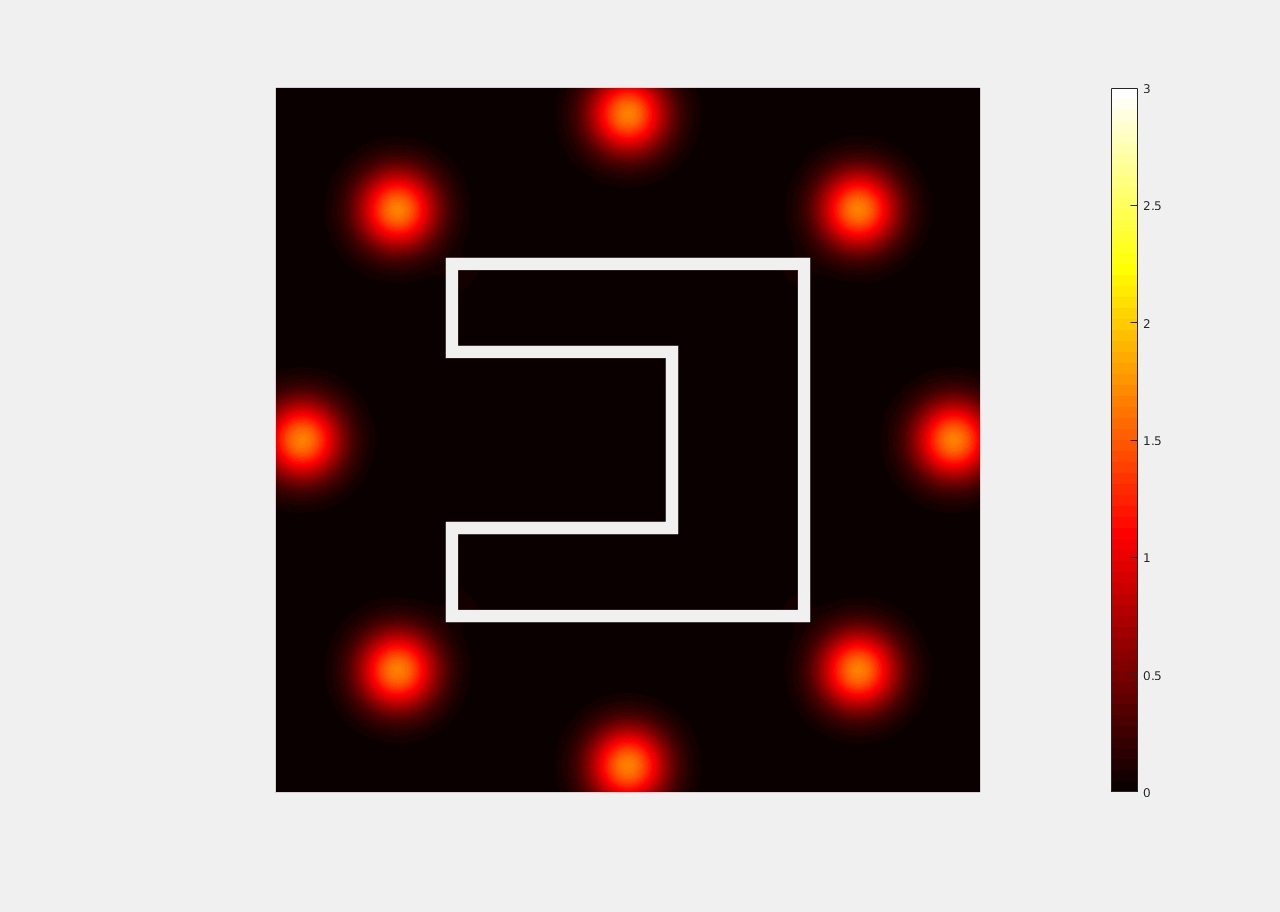}
    \caption{$t=0.06$}
  \end{subfigure}
  \begin{subfigure}{\mywidth}
    \center
    \includegraphics[width=\picwidth]{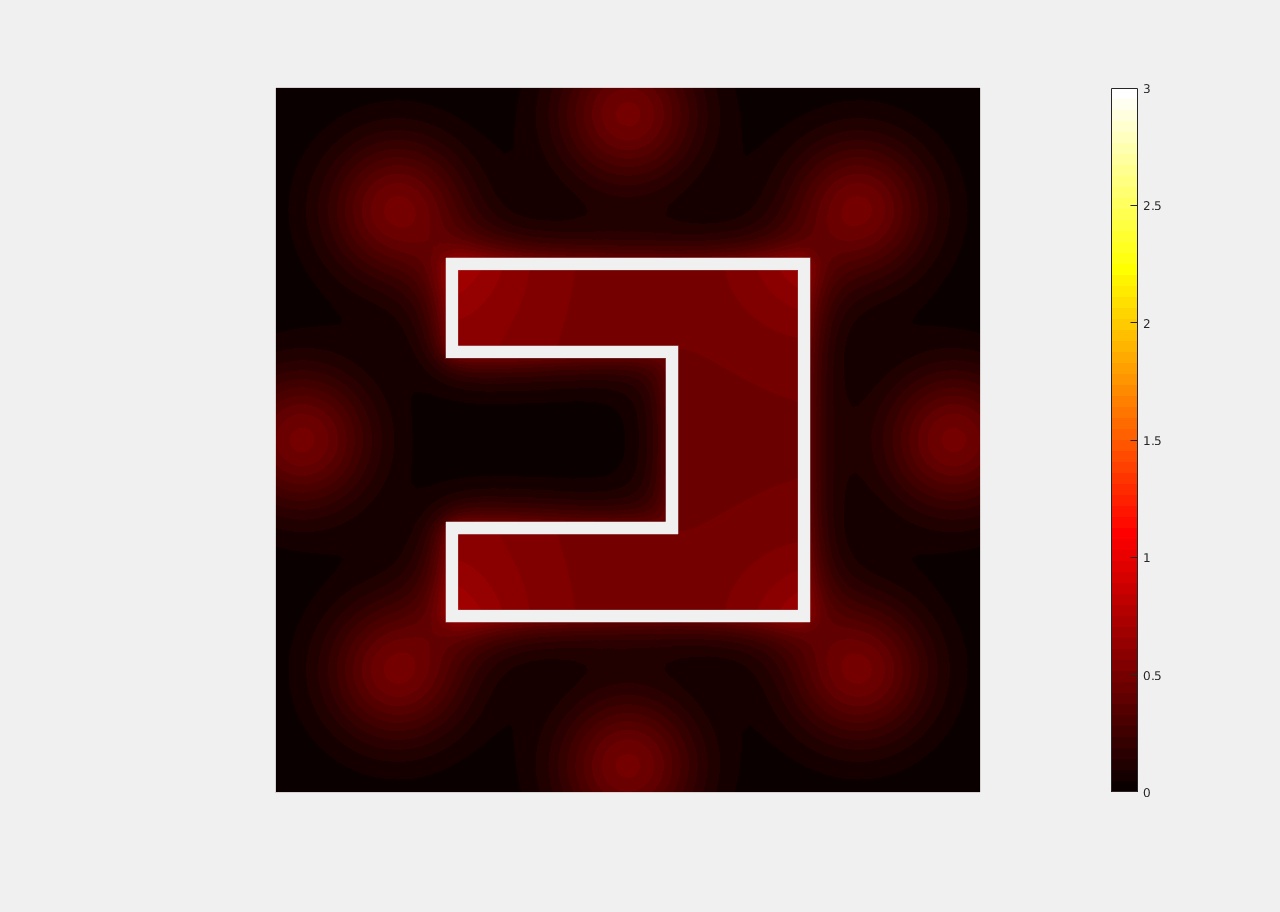}
    \caption{$t=0.2$}
  \end{subfigure}
  \hfill
  \begin{subfigure}{\mywidth}
    \center
    \includegraphics[width=\picwidth]{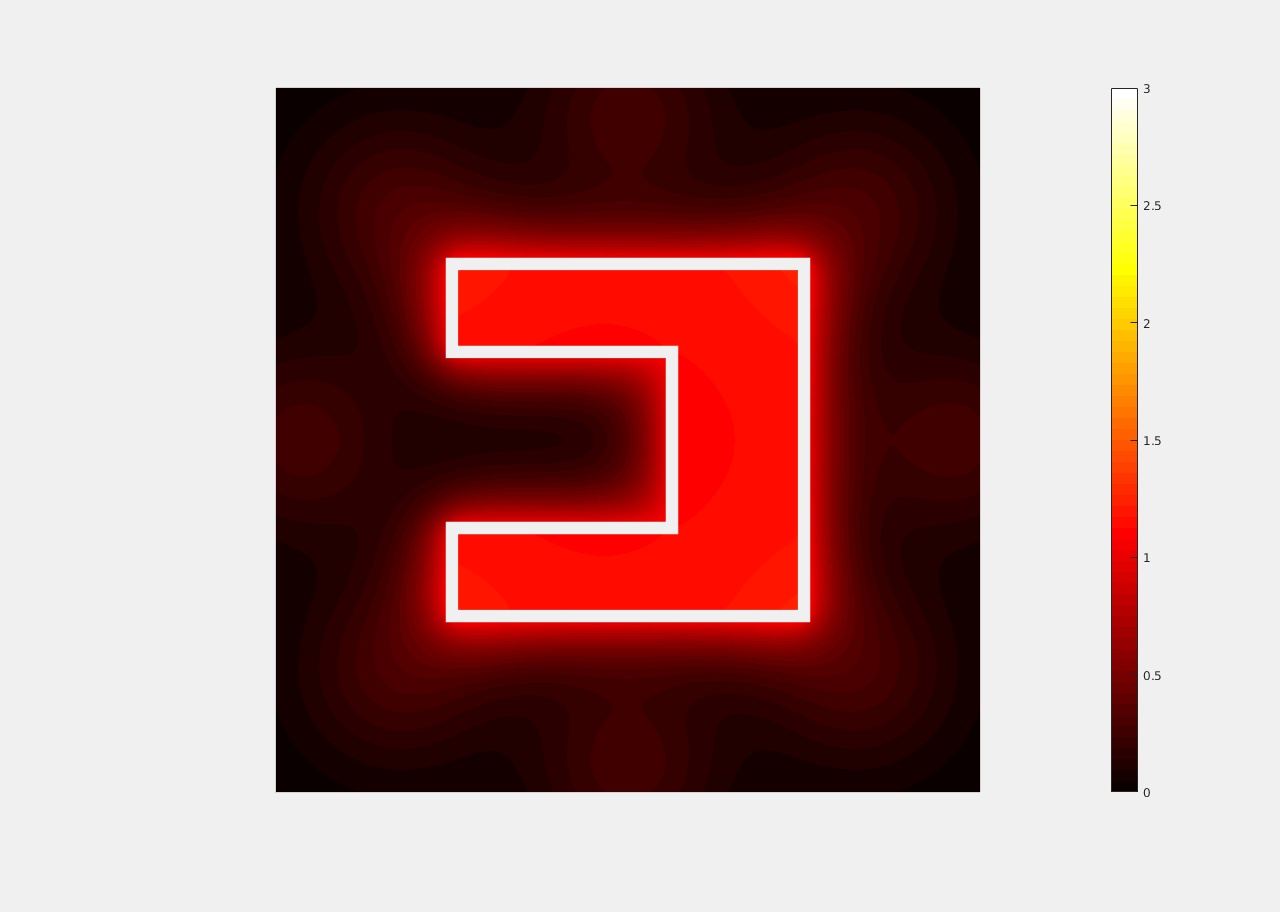}
    \caption{$t=0.375$}
  \end{subfigure}
  \hfill
  \begin{subfigure}{\mywidth}
    \center
    \includegraphics[width=\picwidth]{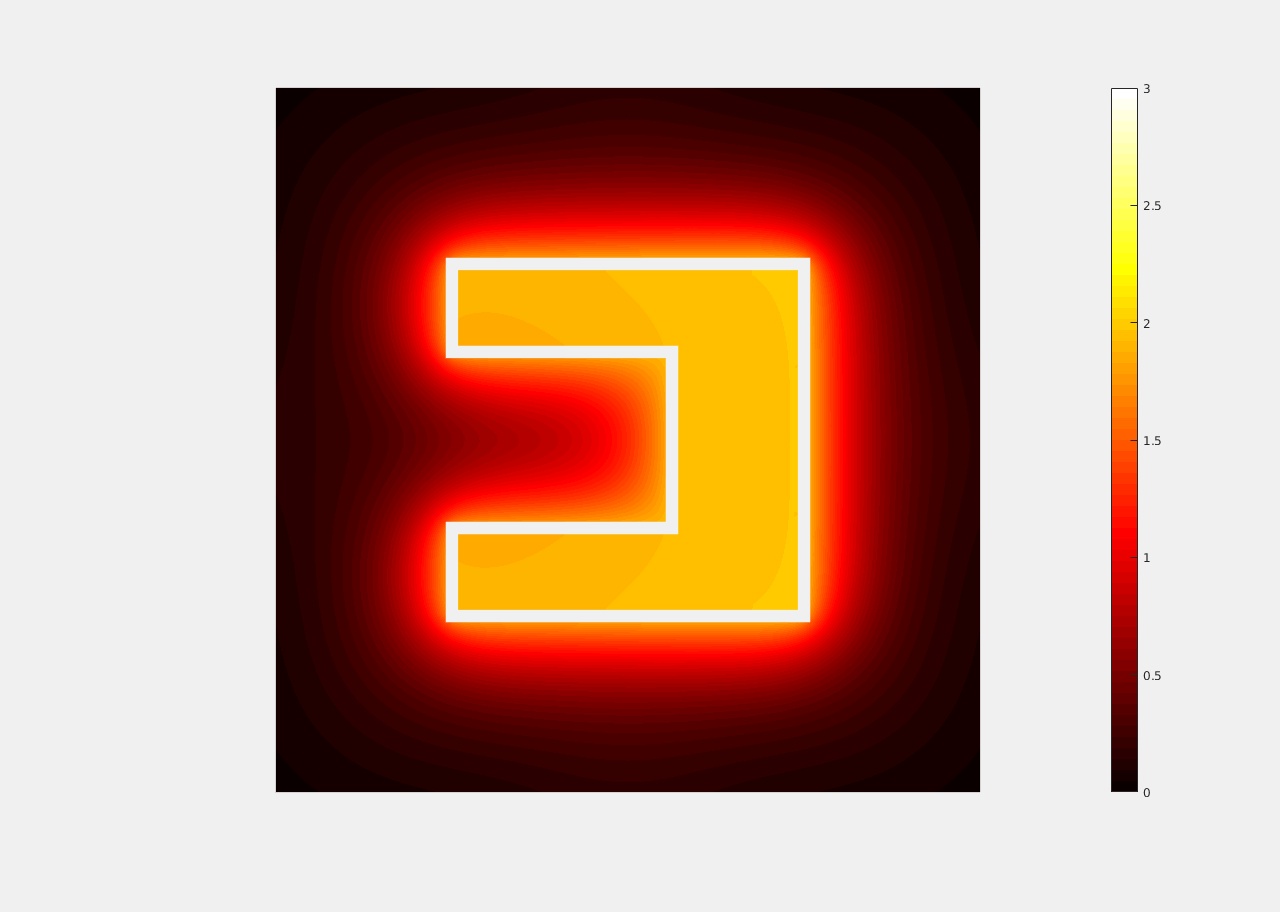}
    \caption{$t=0.75$}
  \end{subfigure}
\hfill
  \begin{subfigure}{\mywidth}
    \center
    \includegraphics[width=\picwidth]{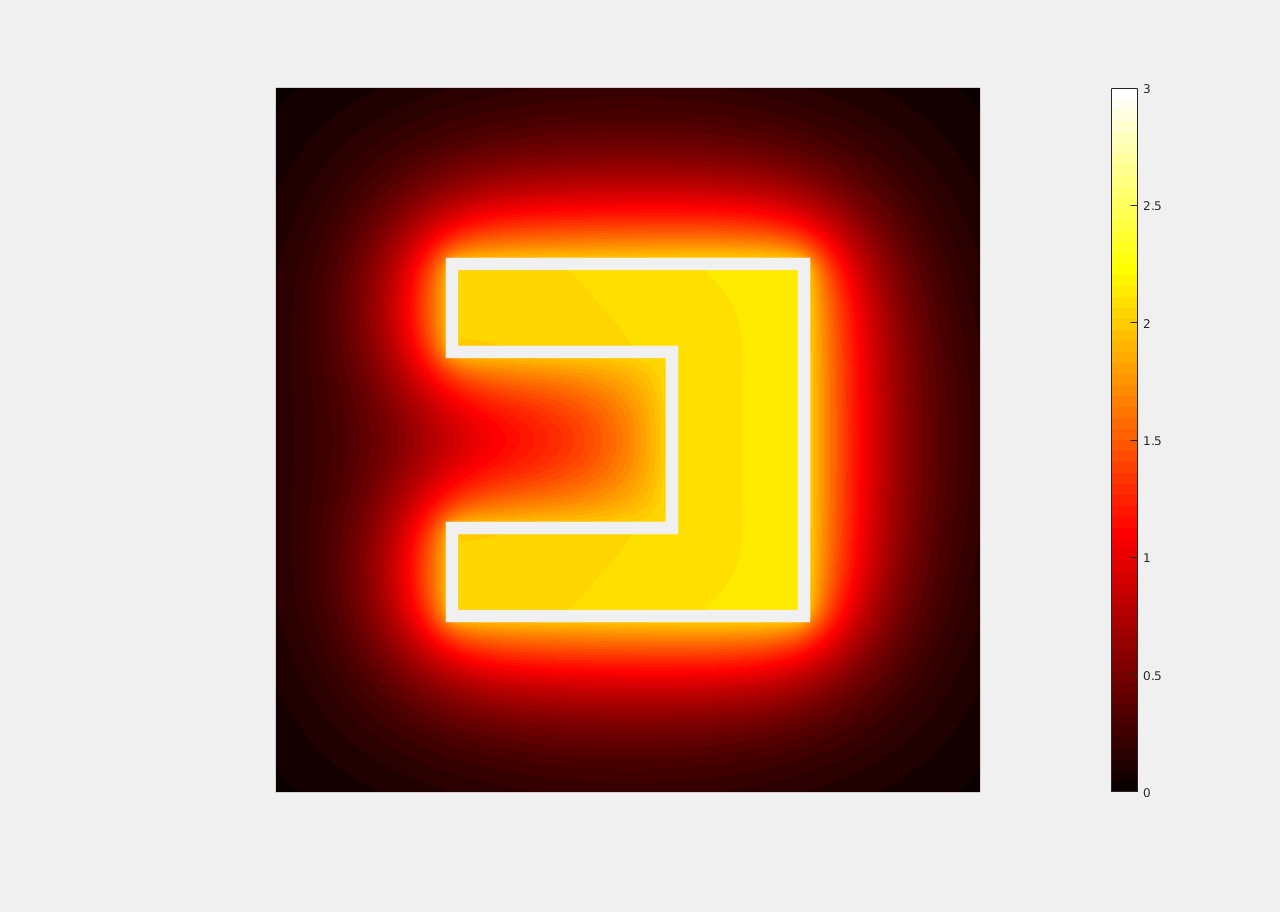}
    \caption{$t=1$}
    \label{fig:simulation:t1}
  \end{subfigure}
  \caption{Time evolution of the heat distribution}
  \label{fig:simulation}
\end{figure}

\section{Conclusions}

We have presented a collection of fully discrete methods for transmission problems associated to the heat equation in free space. The problem is reformulated as a system of time domain boundary integral equations associated to the heat kernel, thus reducing the computation work to the interface between the materials. The system is discretized with Galerkin BEM in the space variable and Convolution Quadrature (of multistep or multistage type) in time. Part of our work has consisted in dealing with the error analysis for the fully discrete method directly in the time domain, thus avoiding typically suboptimal estimates based on Laplace transforms. All the results have been presented for the case of a single inclusion, but the extension to multiple inclusions is straightforward. The case where the inclusion has piecewise constant material properties is more complicated and will be the aim of future work.

\appendix 

\bibliographystyle{abbrv}
\bibliography{Refer}

\section{Background material}\label{sec:A}

\subsection{A result on abstract evolution equations}

\begin{theorem}[Cauchy problems for analytic semigroups] \label{th:semigroup}
Let $X$ be a Hilbert space,  $B:D(B)\to X$ be self-adjoint and maximal dissipative, and let $f\in \mathcal C^\theta(\mathbb{R}_+; X)$ with $\theta\in (0,1)$ and $f(0)=0$. The nonhomogeneous initial value problem
\begin{equation} \label{eq:21}
\dot{w}_0(t) = Bw_0(t) + f(t)\quad t\geq 0,\quad w_0(0) = 0,
\end{equation}
has a unique solution $w_0\in \mathcal C^{1+\theta}(\mathbb{R}_+;X)\cap \mathcal C^\theta (\mathbb{R}_+;D(B))$ and 
\begin{equation} \label{eq:22}
\| w_0(t)\|_X \leq \int_0^t \| f(\tau )\|_X \,\mathrm{d} \tau, \qquad
\| B w_0(t)\|_X \leq  t^\theta \theta^{-1}| f|_{t,\theta, X} + 2\| f(t)\|_X.
\end{equation}
\end{theorem}

\begin{proof}
We will use results from \cite[Section 4.3]{Pazy1983} concerning non-homogeneous problems associated to analytic semigroups. By \cite[Chapter 4, Theorem 3.5(iii)]{Pazy1983} the initial value problem \eqref{eq:21} has a unique solution with the given regularity and this solution is given by the variation of constants formula
\[
w_0(t)=\int_0^t T(t-\tau) f(\tau)\mathrm d\tau,
\]
where $\{ T(t)\}$ is the associated contractive semigroup. Since $B$ is selfadjoint, it follows that $\| tBT(t)\|\le 1$ for all $t>0$ (see \cite[Theorem 4.5.2]{Kesavan1989} for instance). To bound the norm of $Bw_0(t)$ we proceeed as in the proof of \cite[Chapter 4, Theorem 3.2]{Pazy1983} and decompose
\begin{alignat*}{6}
Bw_0(t) & =B\int_0^t T(t-\tau)(f(\tau)- f(t))\,\mathrm{d}\tau + B\int_0^t T(\tau) f(t)\,\mathrm{d}\tau\\
&=\int_0^t B T(t-\tau)(f(\tau)- f(t))\,\mathrm{d}\tau
+T(t) f(t)-f(t).
\end{alignat*}
Since
\[
\int_0^t \| BT(t-\tau) (f(\tau)-f(t))\|_X\mathrm d\tau
\le |f|_{t,\theta,X}\int_0^t\frac1{(t-\tau)^{1-\theta}}\mathrm d\tau,
\]
the result follows easily.
\end{proof}

\subsection{BDF-CQ}\label{app:A2}

Let $\delta$ be the backward differentiation symbol introduced in \eqref{eq:4.1}.
When $\mathrm F:\mathbb C_\star\to \mathcal B(X_1,X_2)$ is an operator-valued analytic function,
\begin{equation}\label{eq:ap1}
\mathrm F(\tfrac1k \delta(\zeta)) = \sum_{n=0}^\infty \omega_n^{\mathrm F} (k) \zeta^n
\end{equation}
is given by the Taylor expansion of the left hand side about $\zeta=0$ since $\delta(\zeta)$ takes values in the domain of $\mathrm F$ for small $|\zeta|$. (This is due to the $A(\theta)$-stability of the BDF formulas of order less than or equal to six.) We thus obtain a sequence of operator-valued coefficients 
\[
\omega_n^{\mathrm F}(k)\in \mathcal B(X_1,X_2).
\]
Then given $\mathrm G(\zeta):=\sum_{n=0}^\infty g_n \zeta^n : \mathbb C_\star \to X_1$, when we multiply
\begin{equation}\label{eq:ap2}
\mathrm F (\tfrac{1}{k} \delta(\zeta)) \mathrm G(\zeta) =
\sum_{n=0}^\infty \sum_{m=0}^n \left( \omega_{n-m}^{\mathrm F}(k) g_m \right) \zeta^n,
\end{equation}
we are just computing the discrete causal convolution of the sequence of operators $\{ \omega_n^{\mathrm F}(k) \}$ to the discrete sequence $\{g_n\}$.
Computational strategies for efficient implementation of multistep CQ (in the context we find it in \eqref{eq:45}-\eqref{eq:46}) can be found in the literature. The lecture notes \cite{HaSa2016} contain a simple introduction to the topic.

\subsection{Rudiments of functional calculus} \label{app:B}

We here introduce some minimun requirements on (Dunford-Riesz) functional calculus needed for our work. First of all, here is a result concerning bounds for rational functions of non-negative self-adjoint operators. The theorem is a consequence of more general results related to functions of operators, which can be found in general introductions to functional calculus (see, for instance, \cite[Corollary 7.1.6, Theorem 2.2.3]{Haase2006}).

\begin{theorem}\label{the:A2}
Let $B:D(B)\subset X\to X$ be a self-adjoint non-negative operator in a Hilbert space $X$. Let $P,Q\in \mathcal P(\mathbb C)$ be two polynomials such that the rational function $R:=P/Q$ is bounded at infinity ($\mathrm{deg}\,P\le \mathrm{deg}\,Q$) and has all its poles in $\{s\in \mathbb C\,:\, \mathrm{Re}\,s<0\}$. Then $R(B):X\to X$ is bounded and
\[
\| R(B)\|_{X\to X}\le \sup_{z>0} |R(z)|.
\]
\end{theorem}

We will also need the evaluation of analytic functions on matrices, using Dunford calculus. Let $\mathrm F:\mathcal O\subset \mathbb C\to X$ be an analytic function defined on a simply connected open set of $\mathbb C$. Let $\mathcal M$ be a matrix whose spectrum is contained in $\mathcal O$. We then define
\[
\mathrm F(\mathcal M):=\frac1{2\pi\imath}
\oint_C (z\mathcal I_s-\mathcal M)^{-1}\otimes \mathrm F(z)\mathrm dz,
\]
where $C$ is any simple positively oriented open contour in $\mathcal O$ surrounding the spectrum of $\mathcal M$.

\subsection{RK-CQ}\label{app:A4}

Let $\mathrm F:\mathbb C_+\to \mathcal B(X_1,X_2)$ be an operator-valued analytic function and consider the expansion \eqref{eq:ap1}, 
where now $\delta(\zeta)\in \mathbb R^{s\times s}$ is defined by \eqref{eq:deltazeta}. Since for $\zeta$ small $\delta(\zeta)$ is a small perturbation of $\mathcal Q^{-1}$ and $\mathcal Q$ has its spectrum contained in $\mathbb C_+$, then $\mathrm F(k^{-1}\delta(\zeta))$ can be defined using functional calculus as in Section \ref{app:B}. The expansion is then a simple Taylor expansion about the origin for an analytic function with values in $\mathcal B(X_1,X_2)^{s\times s}$. The coefficients of this expansion can be given by the Cauchy integrals
\[
\omega_n^{\mathrm F} (k)=\frac1{2\pi\imath}\oint_C (z\mathcal I_s-k^{-1} \mathcal Q^{-1})^{-(n+1)}\otimes \mathrm F(z)\mathrm dz 
\in \mathcal B(X_1,X_2)^{s\times s},
\]
where $C$ is any simple positively oriented closed contour in $\mathbb C_+$ surrounding the spectrum of $\mathcal Q^{-1}$. Discrete causal convolutions in the form \eqref{eq:ap2} can be made now for sequences $\{ g_n\}$ in $X_1^s$. A simple practical introduction to Dunford calculus related to RK-CQ methods can be found in \cite[Section 6.3]{HaSa2016}. Practical computational strategies involve diagonalizing $\delta(\zeta)$, cf. \cite{Banjai2010,BaSc2012,HaSa2016}.

\end{document}